\documentclass[11pt,reqno]{amsart}

\usepackage{amssymb, amsmath, amsthm}
\usepackage{hyperref}
\usepackage[alphabetic,lite]{amsrefs}
\usepackage{verbatim}
\usepackage{amscd}   
\usepackage[all]{xy} 
\usepackage{youngtab} 
\usepackage{young} 
\usepackage{tikz}
\usepackage{ mathrsfs }
\usepackage{cases}

\newcommand{\defi}[1]{{\upshape\sffamily #1}}

\renewcommand{\a}{\alpha}
\renewcommand{\b}{\beta}
\renewcommand{\d}{\delta}
\newcommand{\D}{\mathcal{D}}
\newcommand{\bw}{\bigwedge}
\renewcommand{\det}{\textrm{det}}
\newcommand{\G}{\Gamma}

\renewcommand{\ll}{\lambda}

\newcommand{\om}{\omega}

\newcommand{\oo}{\otimes}

\newcommand{\s}{\sigma}

\newcommand{\Der}{\operatorname{Der}}

\newcommand{\GL}{\operatorname{GL}}

\newcommand{\Lie}{\operatorname{Lie}}

\newcommand{\Pf}{\operatorname{Pf}}
\newcommand{\SL}{\operatorname{SL}}

\newcommand{\Spec}{\operatorname{Spec}}
\newcommand{\Sym}{\operatorname{Sym}}

\newcommand{\codim}{\operatorname{codim}}
\renewcommand{\det}{\operatorname{det}}
\newcommand{\dom}{\operatorname{dom}}
\newcommand{\gr}{\operatorname{gr}}
\newcommand{\htt}{\operatorname{ht}}
\newcommand{\opmod}{\operatorname{mod}}
\newcommand{\op}{\operatorname{op}}
\newcommand{\sdet}{\operatorname{sdet}}
\newcommand{\sgn}{\operatorname{sgn}}
\renewcommand{\skew}{\operatorname{skew}}
\newcommand{\symm}{\operatorname{symm}}

\newcommand{\bb}[1]{\mathbb{#1}}

\renewcommand{\rm}[1]{\textrm{#1}}
\newcommand{\mc}[1]{\mathcal{#1}}
\newcommand{\mf}[1]{\mathfrak{#1}}
\newcommand{\ol}[1]{\overline{#1}}

\newcommand{\tl}[1]{\tilde{#1}}
\newcommand{\ul}[1]{\underline{#1}}
\newcommand{\scpr}[2]{\left\langle #1,#2 \right\rangle}

\def\lra{\longrightarrow}

\newtheorem{theorem}{Theorem}[section]
\newtheorem*{theorem*}{Theorem}
\newtheorem*{problem*}{Problem}
\newtheorem{lemma}[theorem]{Lemma}

\newtheorem{proposition}[theorem]{Proposition}

\newtheorem*{corollary*}{Corollary}

\newtheorem*{chargenl*}{Theorem on Equivariant $\D$-modules on General Matrices}
\newtheorem*{charskewsym*}{Theorem on Equivariant $\D$-modules on Skew-symmetric Matrices}
\newtheorem*{charsym*}{Theorem on Equivariant $\D$-modules on Symmetric Matrices}

\theoremstyle{definition}
\newtheorem{definition}[theorem]{Definition}
\newtheorem*{definition*}{Definition}
\newtheorem{example}[theorem]{Example}

\theoremstyle{remark}
\newtheorem{remark}[theorem]{Remark}
\newtheorem*{remark*}{Remark}

\numberwithin{equation}{section}

\begin{document}

\title{Characters of equivariant $\D$-modules on spaces of matrices}

\author{Claudiu Raicu}
\address{Department of Mathematics, University of Notre Dame, 255 Hurley, Notre Dame, IN 46556\newline
\indent Institute of Mathematics ``Simion Stoilow'' of the Romanian Academy}
\email{craicu@nd.edu}


\date{\today}


\begin{abstract} We compute the characters of the simple $\GL$-equivariant holonomic $\D$-modules on the vector spaces of general, symmetric and skew-symmetric matrices. We realize some of these $\D$-modules explicitly as subquotients in the pole order filtration associated to the determinant/Pfaffian of a generic matrix, and others as local cohomology modules. We give a direct proof of a conjecture of Levasseur in the case of general and skew-symmetric matrices, and provide counterexamples in the case of symmetric matrices. The character calculations are used in subsequent work with Weyman to describe the $\D$-module composition factors of local cohomology modules with determinantal and Pfaffian support.
\end{abstract}

\dedicatory{Dedicated to Jerzy Weyman, on the occasion of his 60th birthday}
\maketitle

\section{Introduction}\label{sec:intro}

When $G$ is an algebraic group acting on a smooth algebraic variety $X$ over $\bb{C}$, it is a natural problem to describe the simple $G$-equivariant holonomic $\D$-modules on $X$. When $G$ acts with finitely many orbits, all such $\D$-modules have regular singularities, and they are classified via the Riemann--Hilbert correspondence by the $G$-equivariant simple local systems on the orbits of the group action. Describing these $\D$-modules explicitly is however a difficult problem (see Open Problem 3 in \cite[Section~6]{mac-vil}, and \cite{vilonen}). In this paper we consider the case when $X$ is a vector space of matrices (general, symmetric, or skew-symmetric), and $G$ is a natural rank preserving group of symmetries. In all these cases $G$ is a reductive group and the $\D$-modules are \defi{$G$-admissible} representations (they decompose into a direct sum of irreducible representations, each appearing with finite multiplicity). The purpose of this paper is to describe these representations (which we will refer to as the \defi{characters} of the equivariant $\D$-modules) and to realize these $\D$-modules explicitly. The motivation for this work is two-fold:
\begin{itemize}
 \item {\bf Computing local cohomology.} In \cites{raicu-weyman,raicu-weyman-witt,raicu-weyman-loccoh} we describe the characters, and the $\D$-module composition factors of the local cohomology modules $\mc{H}_Y^{\bullet}(X,\mc{O}_X)$ in the case when $X$ is a space of matrices (general, symmetric, or skew-symmetric), and $Y$ is any orbit closure for the natural group action on $X$. We expect that the combination of $\D$-module and commutative algebra techniques that we employ to study local cohomology in the case of matrices will apply to other cases of interest \cite[Appendix]{levasseur}. We note that character calculations in the context of analyzing local cohomology modules appear also in \cites{kempf,VdB:loccoh}: in both cases, the representations are $T$-admissible for $T$ a maximal torus in $G$; the equivariant $\D$-modules that we study in this paper are $G$-admissible, but in general they are too large to be $T$-admissible.
 
 \item {\bf Levasseur's conjecture.} For a class of multiplicity-free $G$-representations $X$, Levasseur conjectured \cite[Conjecture~5.17]{levasseur} an equivalence between the category $\mc{C}$ of equivariant holonomic $\D$-modules whose characteristic variety is a union of conormal varieties to the orbits of the group action, and a module category admitting a nice quiver description. His formulation is equivalent to the fact that any simple $\D$-module $\mc{M}$ in $\mc{C}$ contains sections which are invariant under the action of the derived subgroup $G'=[G,G]$. Our character description provides a direct proof of this conjecture for general and skew-symmetric matrices, and yields counterexamples for symmetric matrices.
\end{itemize}
Our work complements the existing literature that studies the categories of $\D$-modules on rank stratifications \cites{nang-nxnmat,nang-skew} (see also \cite{braden-grinberg} for the corresponding categories of perverse sheaves), in that we realize concretely the simple objects of these categories and discuss some applications, filling some gaps in the arguments and generally painting a more transparent picture. To give a flavor of the level of concreteness that we seek, we begin with the following ($\bb{Z}^n_{\dom}$ denotes the set of \defi{dominant weights} $\ll=(\ll_1\geq\cdots\geq\ll_n)\in\bb{Z}^n$, and $S_{\ll}$ denotes the \defi{Schur functor} associated to $\ll$; throughout the paper we use the convention $\ll_s=\infty$ for $s\leq 0$, $\ll_s=-\infty$ for $s>n$):

\begin{theorem}\label{thm:nxnbasicthm}
 Let $X=\bb{C}^{n\times n}$ be the vector space of $n\times n$ matrices, and let $S=\bb{C}[x_{i,j}]$ be the coordinate ring of $X$. If we write $\det=\det(x_{i,j})$, and let $S_{\det}$ be the localization of $S$ at $\det$, then we have a filtration
 \[0\subsetneq S\subsetneq\langle\det^{-1}\rangle_{\D}\subsetneq\cdots\subsetneq\langle\det^{-n}\rangle_{\D}=S_{\det},\]
 where $F_s=\langle \det^{-s}\rangle_{\D}$ denotes the $\D$-submodule of $S_{\det}$ generated by $\det^{-s}$ for $s=0,\cdots,n$ (and $F_{-1}=0$). The successive quotients $A_s=F_s/F_{s-1}$, $s=0,\cdots,n$ are the simple $\GL_n(\bb{C})\times\GL_n(\bb{C})$-equivariant holonomic $\D$-modules on $X$ (for the natural action by row and column operations) and their characters are given by
 \[A_s=\bigoplus_{\substack{\ll\in\bb{Z}^n_{\dom} \\ \ll_s\geq s\geq\ll_{s+1}}}S_{\ll}\bb{C}^n\oo S_{\ll}\bb{C}^n.\]
\end{theorem}

In the case of symmetric matrices, the $\D$-modules obtained as in Theorem~\ref{thm:nxnbasicthm} cover roughly half of the simple equivariant $\D$-modules. The remaining half are more mysterious, and they provide counterexamples to \cite[Conjecture~5.17]{levasseur}. In the case of $m\times n$ matrices with $m>n$, as well as in the case of skew-symmetric matrices of odd size, the simple equivariant $\D$-modules arise as local cohomology modules, while in the case of skew-symmetric matrices of even size the simple equivariant $\D$-modules arise, just as in Theorem~\ref{thm:nxnbasicthm}, from the pole order filtration associated with the Pfaffian of the generic skew-symmetric matrix. Most of our simple $\D$-modules have irreducible characteristic variety, but for roughly half of the ones arising from symmetric matrices the characteristic variety has two connected components: this is deduced in Remark~\ref{rem:charvarieties} as a consequence of the character information.

As suggested by Theorem~\ref{thm:nxnbasicthm}, one motivation behind our investigation is that the simple $\D$-modules are the building blocks for many $\D$-modules of interest that one would like to understand. More precisely, every holonomic $\D$-module $\mc{M}$ has finite length, i.e. it has a finite filtration (\defi{composition series}) whose successive quotients (\defi{composition factors}) are simple holonomic $\D$-modules. When $G$ is connected and $\mc{M}$ is $G$-equivariant, the composition factors are also $G$-equivariant \cite[Prop.~3.1.2]{VdB:loccoh}. We are mainly interested in two types of $G$-equivariant holonomic $\D$-modules:
\begin{itemize}
 \item {\bf Local cohomology modules.} If $Y\subset X$ is a $G$-invariant subset, then the local cohomology modules $\mc{H}_Y^{\bullet}(X,\mc{O}_X)$ are $G$-equivariant $\D$-modules. If $Y$ is smooth and irreducible, and if we write $c=\codim_X(Y)$ for the codimension of $Y$ inside $X$, then $\mc{H}_Y^c(X,\mc{O}_X)$ is the unique non-vanishing local cohomology module and it is simple. In general, for an irreducible subvariety $Y\subset X$ one can define an \defi{intersection homology} $\D$-module $\mc{L}(Y,X)$ which is simple (and it is $G$-equivariant when $Y$ is a $G$-subvariety), and we have an inclusion $\mc{L}(Y,X)\subset\mc{H}_Y^c(X,\mc{O}_X)$, whose cokernel is suported on a proper subset of $Y$. The case when $X=\bb{C}^{n\times n}$ and $Y$ is the subvariety of singular matrices is implicitly described in Theorem~\ref{thm:nxnbasicthm}: $c=1$, $\mc{L}(Y,X)=A_1$, $\mc{H}^1_Y(X,\mc{O}_X)=S_{\det}/S$, and the cokernel $\mc{H}^1_Y(X,\mc{O}_X)/\mc{L}(Y,X)$ has composition factors $A_2,\cdots,A_n$. In general, the local cohomology modules $\mc{H}_Y^i(X,\mc{O}_X)$ for $i\neq c$ may be non-zero, but they are all supported on proper subsets of $Y$: it is an interesting problem to decide their (non)vanishing, or at a more refined level to understand their $\D$-module composition factors.
 
 \item {\bf The $\D$-module (generated by) $f^{\a}$.} For a non-zero polynomial $f\in S=\bb{C}[x_1,\cdots,x_N]$ and a complex number $\a$, we can define $\langle f^{\a}\rangle_{\D}$ -- the (holonomic) $\D$-module generated by $f^{\a}$ (see \cite{walther-fs} for a recent survey). A strict inclusion $\langle f^{\a+1}\rangle_{\D}\subsetneq\langle f^{\a}\rangle_{\D}$ implies that $\a$ is a root of the Bernstein-Sato polynomial of $f$ (this can happen only when $\a$ is rational and negative \cite{kashiwara-bs}). It is an interesting question to decide whether each root $\a$ gives rise to such a strict inclusion \cite[Question~2.1]{walther-fs}, \cite[Question~1,~Section~4]{saito}. More generally, one may be interested in the composition factors of $\langle f^{\a}\rangle_{\D}$. For $\a\in\bb{Z}$ and $f=\det$ this is completely answered by Theorem~\ref{thm:nxnbasicthm}. When $\a\notin\bb{Z}$, $\langle\det^{\a}\rangle_{\D}$ is a simple $\D$-module (see the proof of Theorem~\ref{thm:nonequivLa}). Similar conclusions are obtained when $f$ is the symmetric determinant, or the Pfaffian of a skew-symmetric matrix of even size.
\end{itemize}

Before stating our results in more detail, we give a simple example to illustrate how character calculations alone can allow one to determine the $\D$-module composition factors.
\begin{example}\label{ex:TactsCN}
 Let $X=\bb{C}^N$ be the $N$-dimensional affine space, and let $G=(\bb{C}^*)^N$ be the $N$-dimensional torus. The orbits $X_I$ of the $G$-action are indexed by subsets $I\subset[N]=\{1,\cdots,N\}$, where
 \[X_I=\{x\in\bb{C}^N:x_i\neq 0\rm{ if and only if }i\in I\}.\]
 The stabilizer of each $X_I$ is connected, so there is a one-to-one correspondence between orbits and simple $G$-equivariant holonomic $\D$-modules $D_I$ (Theorem~\ref{thm:equivRH}, Remark~\ref{rem:Ltrivial}), given by $D_I=\mc{L}(Y_I,X)$, where $Y_I=\ol{X_I}$ is the corresponding orbit closure. Since $Y_I$ is an affine space of codimension $N-|I|$, it is in particular smooth, and therefore the $\D$-module $D_I$ is just a local cohomology module $D_I=H_{Y_I}^{N-|I|}(X,\mc{O}_X)$. If we write $S=\bb{C}[x_1,\cdots,x_N]$ for the coordinate ring of $X$, then each $Y_I$ is defined by the ideal generated by the variables $x_j$, $j\notin I$. Using the \v Cech complex description of local cohomology we get
 \[D_I=\bigoplus_{\substack{\ll\in\bb{Z}^N \\ \ll_i\geq 0\rm{ if and only if }i\in I}} \bb{C}\cdot x_1^{\ll_1}\cdots x_N^{\ll_N}\]
 which is a decomposition into irreducible $G$-representations. If we take $f=x_1\cdots x_N$ then we get
 \[S_f = \bigoplus_{\ll\in\bb{Z}^N}\bb{C}\cdot x_1^{\ll_1}\cdots x_N^{\ll_N}.\]
 The torus weights appearing in the $D_I$'s form a partition of those appearing in $S_f$, so each $D_I$ appears as a $\D$-module composition factor of $S_f$ with multiplicity one. Using a similar argument for $X=\bb{C}^{n\times n}$ we obtain a proof of Theorem~\ref{thm:nxnbasicthm} (see Section~\ref{sec:mxnmatrices}).
\end{example}

\subsection*{Symmetric matrices}

Our results run in parallel for the three spaces of matrices (general, symmetric, and skew-symmetric). We have made the effort to apply a uniform strategy to all three cases, but we weren't able to treat the combinatorial details uniformly. For the sake of brevity, we have chosen to treat only the case of symmetric matrices in full detail, and only indicate the changes that are required in the other cases. Two features that make the case of symmetric matrices more interesting are: (a) the presence of non-trivial equivariant local systems; (b) the existence of counterexamples to Levasseur's conjecture.

For each positive integer $n$ and for $s=0,\cdots,n$, we consider the collections of dominant weights
\begin{equation}\label{eq:defmcC^isn}
\begin{aligned}
\mc{C}^1(s,n)&=\{\ll\in\bb{Z}^{n}_{\dom}:\ll_i\overset{(\opmod\ 2)}{\equiv} s+1\rm{ for }i=1,\cdots,n,\ll_{s}\geq s+1\geq\ll_{s+2}\},\\
\mc{C}^2(s,n)&=\left\{\ll\in\bb{Z}^{n}_{\dom}:\ll_i\overset{(\opmod\ 2)}{\equiv}
\begin{cases}
 s+1 & \rm{for }i=1,\cdots,s\\
 s & \rm{for }i=s+1,\cdots,n
\end{cases},
\ll_{s}\geq s+1,\ll_{s+1}\leq s
\right\}.
\end{aligned}
\end{equation}
Note that $\mc{C}^1(n,n)=\mc{C}^2(n,n)$. For a positive integer $n$, we identify $\Sym^2\bb{C}^n$ with the vector space $M^{\symm}$ of $n\times n$ symmetric matrices, where the squares $w^2$, $w\in\bb{C}^n$, correspond to matrices of rank at most one. We write $M^{\symm}_i$ for the subvariety of matrices of rank at most $i$. For $s=0,\cdots,n$, and $j=1,2$, we define
\[\mf{C}^j_s=\bigoplus_{\ll\in\mc{C}^j(s,n)} S_{\ll}\bb{C}^n.\]

\begin{charsym*}[Section~\ref{sec:symm}]
There exist $(2n+1)$ simple $\GL_n(\bb{C})$-equivariant holonomic $\D$-modules on $M^{\symm}$, whose characters are $\mf{C}^j_s$, $s=0,\cdots,n$, $j=1,2$. More precisely, if we denote by $C^j_s$ the $\D$-module with character $\mf{C}^j_s$ then $C^1_n=C^2_n=\mc{L}(\{0\},M^{\symm})$ is the simple holonomic $\D$-module supported at the origin, and for $s<n$
\[
 C^j_s=\begin{cases}
 \mc{L}(M^{\symm}_{n-s},M^{\symm}) & \rm{if }j\equiv s\ (\opmod\ 2), \\
 \mc{L}(M^{\symm}_{n-s},M^{\symm};1/2) & \rm{if }j\equiv s+1\ (\opmod\ 2). \\
\end{cases}
\]
Here $\mc{L}(M^{\symm}_{n-s},M^{\symm})$ is the usual intersection homology $\D$-module, while $\mc{L}(M^{\symm}_{n-s},M^{\symm};1/2)$ is the intersection homology $\D$-module associated to the non-trivial irreducible $\GL_n(\bb{C})$-equivariant local system on the orbit of rank $(n-s)$ matrices.

We let $S=\bb{C}[x_{i,j}]$ be the coordinate ring of $M^{\symm}$, where $x_{i,j}=x_{j,i}$. We write $\sdet=\det(x_{i,j})$ for the determinant of the generic symmetric matrix, and let $S_{\sdet}$ be the localization of $S$ at $\sdet$. We consider $F_s=\langle \sdet^{-s/2}\rangle_{\D}$, the $\D$-submodule of $S_{\sdet}$ (or of $S_{\sdet}\cdot\sdet^{1/2}$) generated by $\sdet^{-s/2}$ for $s=0,\cdots,n+1$ (and $F_{-1}=0$). We have that $C^2_0=F_0=S$, and $C^1_s=F_{s+1}/F_{s-1}$ for $s=0,\cdots,n$.
\end{charsym*}

\begin{remark}\label{rem:failureLev}
 The $\D$-modules $C^2_s$ for $s=1,\cdots,n-1$ contain no $\SL_n(\bb{C})$-invariant sections, so they provide counterexamples to \cite[Conjecture~5.17]{levasseur}. It may be interesting to note that when $n\geq 3$, among these counterexamples there are the intersection homology $\D$-modules $\mc{L}(M^{\symm}_{n-s},M^{\symm})$ with $s$ even, so the failure of Levasseur's conjecture can't be solely explained by the presence of non-trivial local systems!
\end{remark}

\begin{remark}\label{rem:bsato}
 We can now give a quick derivation for the \defi{Bernstein-Sato polynomial} of $\sdet$ \cite[Appendix]{kimura}:
\begin{equation}\label{eq:bsatosdet}
b_{\sdet}(s)=\prod_{i=1}^n \left(s+\frac{1+i}{2}\right).
\end{equation}
It follows from Cayley's identity that $b_{\sdet}(s)$ divides $\prod_{i=1}^n \left(s+\frac{1+i}{2}\right)$, while for each $i=1,\cdots,n$ the strict inclusion $F_{i-1}\subsetneq F_{i+1}$ shows that $-\frac{1+i}{2}$ is a root of $b_{\sdet}(s)$. This is enough to conclude the equality (\ref{eq:bsatosdet}).
\end{remark}

\begin{remark}\label{rem:charvarieties}
 It is interesting to note that the character calculation allows us to determine the characteristic varieties for the $\D$-modules $C_s^j$. The Fourier transform $\mc{F}$ (see Section~\ref{subsec:Fourier}) permutes the $\D$-modules $C_s^j$, and ``rotates'' their characteristic varieties by $90^{\circ}$ (note that ``rotating'' the conormal variety to the orbit of rank $s$ matrices yields the conormal variety to rank $(n-s)$ matrices). The formula (\ref{eq:defGFourierU}) where $U=\bw^2\bb{C}^n$, together with (\ref{eq:defmcC^isn}), shows that $\mc{F}(C_s^1)=C_{n-s-1}^1$ for $s=0,\cdots,n-1$, and $\mc{F}(C_s^2)=C_{n-s}^2$ for $s=0,\cdots,n$. Since $C_s^1$ has support $M^{\symm}_{n-s}$ and $\mc{F}(C_s^1)$ has support $M^{\symm}_{s+1}$, it follows that the characteristic variety of $C_s^1$ has two components, namely the conormal varieties to the orbits of rank $(n-s)$ and rank $(n-s-1)$ matrices. Since $C_s^2$ has support $M^{\symm}_{n-s}$ and $\mc{F}(C_s^2)$ has support $M^{\symm}_s$, it follows that the characteristic variety of $C_s^2$ is irreducible, namely it is the conormal variety to the orbit of rank $(n-s)$ matrices. Similar considerations show that for general and skew-symmetric matrices, the characteristic varieties of the simple equivariant $\D$-modules are irreducible. The calculation of characteristic varieties can also be deduced from~\cite{braden-grinberg}.
\end{remark}

\subsection*{Strategy for computing the characters of equivariant $\D$-modules}\label{subsec:computingcharacters}

Our approach to computing characters of equivariant $\D$-modules is based on performing Euler characteristic calculations using the $\D$-module functoriality together with some combinatorial and geometric methods. More precisely, for the inclusion of an orbit $\iota:O\hookrightarrow X$, the $\D$-module direct image $\int_{\iota}\mc{O}_O$ is an object in the derived category of $G$-equivariant $\D_X$-modules, whose cohomology groups $\int^j_{\iota}\mc{O}_O$ are (in the cases that we study) $G$-admissible representations. Analyzing the inclusion $\iota$ directly is complicated, so we make use of a resolution of singularities $Z$ of the orbit closure $\ol{O}$. The variety $Z$ is a vector bundle over a Grassmannian $\bb{G}$ (or a product of Grassmannians), and the inclusion $j:O\hookrightarrow Z$ is an affine open immersion. The map $\pi:Z\to X$ factors as $p\circ s$
\[
\xymatrix{
O \ar@{^{(}->}[r]^{j} \ar@{=}[d] & Z \ar@{^{(}->}[r]^{s} \ar[d] \ar[dr]_{\pi} & X\times\bb{G} \ar[d]_{p} \\
O \ar@{^{(}->}[r] & \ol{O} \ar@{^{(}->}[r] & X \\
} 
\]
where $s$ is a regular embedding and $p$ is the projection onto the first factor. We compute the Euler characteristic of $\int_{\iota}\mc{O}_O$ as a virtual admissible $G$-representation, by using the factorization $\iota=p\circ s\circ j$. If we pretend that there is a one-to-one correspondence between simple equivariant $\D$-modules and orbits (which is true for general and skew-symmetric matrices), and write $X_s$ for the $\D$-module corresponding to matrices of rank $s$, then the Euler characteristic calculations together with general considerations regarding the structure of $\D$-module direct images, allow us to write down an upper-triangular matrix with ones on the diagonal, that represents the change of coordinates in the Grothendieck group of admissible representations, from $(X_s)_s$ to appropriately defined linearly independent characters $(\mf{X}_s)_s$. The Fourier transform on one hand preserves this matrix, and on the other hand it makes it lower-triangular, which allows us to conclude that the matrix is in fact the identity and therefore $X_s=\mf{X}_s$ for all $s$ (see Section~\ref{subsec:linalg}).

In the process of computing Euler characteristics, we are led to the following combinatorial problem. Let $X=\bb{G}(k,\bb{C}^n)$ be the Grassmannian of $k$-dimensional quotients of $\bb{C}^n$, with $\mc{O}_X(1)$ denoting the Pl\"ucker line bundle, and $\Omega^i_X$ denoting the sheaf of differential $i$-forms on $X$, and define the virtual $\GL_n(\bb{C})$-representation
\[p_{k,r}=\sum_{i=0}^{k\cdot(n-k)}(-1)^i\cdot\chi(X,\Omega_X^i(r)).\]
The problem is to compute $p_{k,r}\oo S_{\ll}\bb{C}^n$ for $\ll\in\bb{Z}^n_{\dom}$. When $k=1$, $p_{1,r}$ corresponds to the $r$-th power sum symmetric function, and the answer is given in \cite[Exercise~I.3.11(1)]{macdonald}. The relevance of this formula for computing Euler characteristics is as follows: if we write $E=\mc{L}(\{0\},X)$ for the simple holonomic $\D$-module supported at the origin, $O_k$ for the orbit of rank $k$ matrices, and $\iota_k$ for the inclusion of $O_k$ into the ambient space, then (up to minor adjustments, depending on which space of matrices we analyze)
\[\chi\left(\int_{\iota_k}\mc{O}_{O_k}\right)=\sum_{j\in\bb{Z}}(-1)^j\int_{\iota_k}^j\mc{O}_{O_k}=\lim_{r\to\infty}p_{k,r}\oo E,\]
where the limit is taken in the Grothendieck group of admissible representations (see Section~\ref{subsubsec:admissible} for a precise formulation, and Section~\ref{sec:lims} for the calculations).

\subsection*{Organization}

In Section~\ref{sec:prelim} we establish the notation and basic results concerning the representation theory of general linear groups and $\D$-modules that will be used throughout the rest of the paper. In Section~\ref{sec:lims} we compute the relevant Euler characteristics as limits in the Grothendieck group of $\GL$-admissible representations. In Sections~\ref{sec:symm},~\ref{sec:mxnmatrices}, and~\ref{sec:skew} we prove the main results on characters of equivariant $\D$-modules. Finally, in Section~\ref{sec:simplesrankstrat} we discuss the simple $\D$-modules that arise from non-equivariant local systems on the orbits, and prove Levasseur's conjecture for skew-symmetric and general matrices.

\section{Preliminaries}\label{sec:prelim}

\subsection{Representation Theory {\cite[Ch.~2]{weyman}}}\label{subsec:repthy}

Let $W$ be a complex vector space of dimension $\dim(W)=n$, and denote by $\GL(W)$ the group of invertible linear transformations of $W$. The irreducible finite dimensional $\GL(W)$-representations, denoted $S_{\ll}W$, are indexed by \defi{dominant weights} $\ll=(\ll_1\geq\cdots\geq\ll_n)\in\bb{Z}^n$. A dominant weight $\ll$ is said to be a \defi{partition} if all its \defi{parts} $\ll_1,\cdots,\ll_n$ are nonnegative. The \defi{size} of $\ll$ is $|\ll|=\ll_1+\cdots+\ll_n$. The \defi{conjugate partition} $\ll'$ is defined by transposing the associated Young diagram: $\ll'_i$ is the number of $j$'s for which $\ll_j\geq i$; for example $(5,2,1)'=(3,2,1,1,1)$. Write $[n]$ for the set $\{1,\cdots,n\}$, and for a given a subset $I\subset[n]$ and an integer $u$, let $(u^I)$ be the sequence $\mu\in\bb{Z}^n$ having $\mu_i=u$ when $i\in I$, and $\mu_i=0$ when $i\notin I$. When $I=[k]$ for $k\leq n$, we simply write $(u^k)$ instead of $(u^I)$. We have that $S_{(1^k)}W=\bw^k W$ is the $k$-th exterior power of $W$, and we let $\det(W)$ denote the top exterior power $\bw^n W$.

\subsubsection{Admissible representations}\label{subsubsec:admissible}

Given a reductive algebraic group $G$, we write $\Lambda$ for the set of (isomorphism classes of) finite dimensional irreducible $G$-representations. We will be mainly interested in the case when $G=\GL(W)$ is a general linear group: we write $\G(G)=\G(W)$, and $\Lambda=\{S_{\ll}W:\ll\in\bb{Z}^n_{\dom}\}$. We also consider $G=\GL(W_1)\times\GL(W_2)$, $\dim(W_1)=m$, $\dim(W_2)=n$, and write $\G(G)=\G(W_1,W_2)$ and $\Lambda=\{S_{\d}W_1\oo S_{\ll}W_2:\d\in\bb{Z}^m_{\dom},\ll\in\bb{Z}^n_{\dom}\}$. An \defi{admissible $G$-representation} decomposes as
\[M=\bigoplus_{L\in\Lambda} L^{\oplus a_L},\]
where each $a_L\in\bb{Z}_{\geq 0}$. We say that $M$ is \defi{finite} if only finitely many of the $a_L$'s are non-zero. We define the \defi{Grothendieck group $\G(G)$ of admissible representations} to be $\bb{Z}^{\Lambda}$, the direct product of copies of $\bb{Z}$, indexed by the set $\Lambda$. We call the elements of $\G(G)$ \defi{virtual representations}. We write a typical element $U\in\G(G)$ as
\[U=\sum_{L\in\Lambda} a_{L}\cdot L,\]
where $a_L\in\bb{Z}$ and define $\scpr{U}{L}=a_L$ to be the \defi{multiplicity} of $L$ inside $U$. A sequence $(U_r)_r$ of virtual representations is said to be \defi{convergent} (in $\G(G)$) if for every $L\in\Lambda$, the sequence of integers $\scpr{U_r}{L}$ is eventually constant. If $(U_r)_r$ is convergent, we write $a_L=\lim_{r\to\infty}\scpr{U_r}{L}$ for each $L\in\Lambda$. We define $U=\sum_{L\in\Lambda}a_L\cdot L$ to be \defi{the limit} of $(U_r)_r$, and write
\[\lim_{r\to\infty}U_r=U.\]

\subsubsection{Combinatorics of weights}\label{subsubsec:combinatoricsweights}
It will be convenient to make sense of $S_{\ll}W$ even when $\ll\in\bb{Z}^n$ is not dominant. In order to do so we let $\d=(n-1,n-2,\cdots,1,0)$ and consider $\ll+\d=(\ll_1+n-1,\ll_2+n-2,\cdots,\ll_{n-1}+1,\ll_n)$. We write $\rm{sort}(\ll+\d)$ for the sequence obtained by rearranging the entries of $\ll+\d$ in non-increasing order. If $\ll+\d$ has non-repeated entries, we let $\sgn(\ll)$ denote the sign of the unique permutation realizing the sorting of the sequence $\ll+\d$. We define
\[\tl{\ll}=\rm{sort}(\ll+\d)-\d,\]
and let $S_{\ll}W$ be the element of $\G(W)$ defined by
\begin{equation}\label{eq:defSll}
S_{\ll}W = \begin{cases}
\sgn(\ll)\cdot S_{\tl{\ll}}W & \rm{if }\tl{\ll}\rm{ is dominant (i.e. if }\ll+\d\rm{ has non-repeated entries);} \\
0 & \rm{otherwise}.            
\end{cases}
\end{equation}
For example, we have $S_{(2,1,4,3)}W=0$ and $S_{(1,1,0,7)}W=-S_{(4,2,2,1)}W$. Note that in particular
\begin{equation}\label{eq:vanishSll}
S_{\ll}W=0\rm{ if }\ll_{i+1}=\ll_i+1\rm{ for some }i=1,\cdots,n-1. 
\end{equation}

We denote by ${[n]\choose k}$ the collection of subsets $I\subset[n]$ of size $|I|=k$, and write $P(k,n-k)$ for the set of partitions $\mu=(\mu_1\geq\cdots\geq\mu_k)$ with $\mu_1\leq n-k$. There is a one-to-one correspondence between sets $I\in{[n]\choose k}$ and partitions $\mu\in P(k,n-k)$ given by
\begin{equation}\label{eq:Itomu}
I = \{\mu_k+1,\mu_{k-1}+2,\cdots,\mu_2+(k-1),\mu_1+k\}. 
\end{equation}
If we write $\mu'\in P(n-k,k)$ for the conjugate partition of $\mu$ then the complement of $I$ in $[n]$ is given by
\begin{equation}\label{eq:Icfrommu'}
I^c=[n]\setminus I=\{k+1-\mu'_1,k+2-\mu'_2,\cdots,n-\mu'_{n-k}\}. 
\end{equation}
For every $\ll\in\bb{Z}^n$, $I\in{[n]\choose k}$ and $r\in\bb{Z}$, we define $\ll(r,I)\in\bb{Z}^n$ as follows: we write the elements of $I$ and $I^c$ in increasing order
\begin{equation}\label{eq:listIIc}
I=\{i_1<\cdots<i_k\},\ I^c=\{i_{k+1}<\cdots<i_{n}\}, 
\end{equation}
and let
\begin{equation}\label{eq:defllrI}
 \ll(r,I)_t=\begin{cases}
 r+t+\ll_{i_t}-i_t & \rm{for }t=1,\cdots,k; \\
 t+\ll_{i_t}-i_t & \rm{for }t=k+1,\cdots,n.
\end{cases}
\end{equation}
 We define $\ll^1(I)\in\bb{Z}^k$ and $\ll^2(I)\in\bb{Z}^{n-k}$ via
\begin{equation}\label{eq:defll12}
\begin{aligned}
 \ll^1(I)_t &= t+\ll_{i_t}-i_t,\rm{ for }t=1,\cdots,k, \\
 \ll^2(I)_{t-k} &= t+\ll_{i_t}-i_t,\rm{ for }t=k+1,\cdots,n,
\end{aligned}
\end{equation}
so that $\ll(r,I)$ is the concatenation of $\ll^1(I)+(r^k)$ and $\ll^2(I)$. In particular
\[\ll^1([k])=(\ll_1,\cdots,\ll_k)\rm{ and }\ll^2([k])=(\ll_{k+1},\cdots,\ll_n).\]
We define the permutation $\s(I)$ of $[n]$ via
\begin{equation}\label{eq:defsI}
 \s(I)_t=i_t\rm{ for }t=1,\cdots,n. 
\end{equation}
With this notation we obtain
\begin{equation}\label{eq:Sll+rI=Sll(rI)}
 S_{\ll+(r^I)}W=\sgn(\s(I))\cdot S_{\ll(r,I)}W=(-1)^{|\mu|}\cdot S_{\ll(r,I)}W,
\end{equation}
and note that if $\ll$ is dominant and $r$ is sufficiently large then $\ll(r,I)$ is also dominant.

We define for $h,j\in\bb{Z}/2\bb{Z}$ the sets of partitions
\begin{equation}\label{eq:defPhjab}
 P^{h,j}(a,b)=\{\mu\in P(a,b):\mu_i\equiv h\ (\opmod\ 2)\rm{ for }i=1,\cdots,a,\mu'_i\equiv j\ (\opmod\ 2)\rm{ for }i=1,\cdots,b\}.
\end{equation}
A quick counting argument yields the following:

\begin{lemma}\label{lem:countpartitions}
 The cardinality of $P^{h,j}(a,b)$ is computed by:
\[
|P^{0,0}(a,b)|={\lfloor \frac{a}{2}\rfloor + \lfloor \frac{b}{2}\rfloor \choose \lfloor \frac{b}{2}\rfloor},\quad
|P^{0,1}(a,b)|=\begin{cases}
\displaystyle{\lfloor \frac{a-1}{2}\rfloor + \frac{b}{2} \choose \frac{b}{2}} & b\rm{ even} \\
0 & b\rm{ odd}
\end{cases},\]
\[
|P^{1,0}(a,b)|=\begin{cases}
\displaystyle{\frac{a}{2} + \lfloor \frac{b-1}{2}\rfloor \choose \lfloor \frac{b-1}{2}\rfloor} & a\rm{ even} \\
0 & a\rm{ odd}
\end{cases},\quad
|P^{1,1}(a,b)|=\begin{cases}
\displaystyle{\frac{a-1}{2} + \frac{b-1}{2} \choose \frac{b-1}{2}} & a,b\rm{ odd} \\
0 & \rm{otherwise}
\end{cases}.
\]
\end{lemma}

\subsubsection{A generalized Pieri rule}\label{subsubsec:genlPieri}
The Grothendieck group $\G(W)$ is a module over the \defi{representation ring $R(W)$} of finite dimensional $\GL(W)$-representations. As a ring, $R(W)$ is generated by the exterior powers $\bw^k W$, $k\leq n$, and by the inverse $\det(W)^{-1}=\bw^n W^*=S_{(-1^n)}W$ of $\det(W)$. We have $S_{\ll}W\oo\det(W)=S_{\ll+(1^n)}W$. The following lemma generalizes this by describing the multiplicative action of the exterior powers $\bw^k W$ on $\G(W)$ (since the multiplication is continuous, i.e. it commutes with limits, it suffices to determine its action on the indecomposables $S_{\ll}W$):

\begin{lemma}[Pieri's rule]\label{lem:Pieri}
For every $\ll\in\bb{Z}^n$ we have the following equality in $\G(W)$:
\begin{equation}\label{eq:Pieri}
\left(\bw^k W\right)\oo S_{\ll}W=\sum_{I\in{[n]\choose k}}S_{\ll+(1^I)}W.
\end{equation}
\end{lemma}

\begin{proof}
 We may assume without loss of generality that $\ll$ is dominant. If $\ll_{i+1}=\ll_i$ and $I$ is such that $i\notin I$ and $(i+1)\in I$ then it follows from (\ref{eq:vanishSll}) that $S_{\ll+(1^I)}W=0$. For all the other terms appearing on the right hand side of (\ref{eq:Pieri}) we have that $\mu=\ll+(1^I)$ is dominant and $\mu/\ll$ is a \defi{vertical strip} (i.e. $\mu_i-\ll_i\in\{0,1\}$ for all $i$) of size $k$. (\ref{eq:Pieri}) then follows from the usual Pieri formula \cite[Corollary~2.3.5]{weyman}.
\end{proof}

We define elements $p_{k,r}(W)\in R(W)$ for every $r\in\bb{Z}$ and $0\leq k\leq n$, by
\begin{equation}\label{eq:defpkr}
 p_{k,r}(W)=\sum_{I\in{[n]\choose k}} S_{(r^I)}W,
\end{equation}
and note that $p_{k,1}(W)=\bw^k W$. We have the following generalization of Pieri's rule:

\begin{lemma}\label{lem:genlPieri}
For every $\ll\in\bb{Z}^n$ we have the following equality in $\G(W)$:
\begin{equation}\label{eq:genlPieri}
p_{k,r}(W)\oo S_{\ll}W=\sum_{I\in{[n]\choose k}}S_{\ll+(r^I)}W.
\end{equation}
\end{lemma}

\begin{proof} When $k=0$, $p_{0,r}(W)=\bb{C}$ is the identity element of $R(W)$, so the conclusion is trivial. We may thus assume that $k>0$. As before, we also assume that $\ll$ is dominant. Multiplication by $\det(W)$ is an invertible operation, so proving (\ref{eq:genlPieri}) for $\ll$ is equivalent to proving it for $\ll+(1^{n})$. In particular, we may assume that $\ll$ is a partition and that moreover $\ll_n = 0$.

We consider the ordering of the partitions $\ll$ with at most $n$ parts induced by the graded reverse lexicographic order on their conjugates: more precisely, we say that $\ll\succ\mu$ if $|\ll|>|\mu|$, or if $|\ll|=|\mu|$ and for the largest index $i$ for which $\ll'_i\neq\mu'_i$ one has $\ll'_i>\mu'_i$. We prove (\ref{eq:genlPieri}) for all partitions $\ll$, by induction with respect to the said ordering. When $\ll$ is the empty partition, (\ref{eq:genlPieri}) coincides with (\ref{eq:defpkr}). 

Assume now that $\ll_1>0$ and consider the parititon $\mu$ obtained from $\ll$ by removing the last column of its Young diagram: the conjugate $\mu'$ is given by $\mu'_i=\ll'_i$ for $i<\ll_1$ and $\mu'_i=0$ for $i\geq\ll_1$. We let $l=\ll'_{\ll_1}$ denote the size of the column removed from $\ll$. Using the induction hypothesis for $\mu\prec\ll$ and Lemma~\ref{lem:Pieri} we get
\begin{equation}\label{eq:pkrSmubw^l}
(p_{k,r}(W)\oo S_{\mu}W)\oo\left(\bw^l W\right)=\sum_{J\in{[n]\choose l}}\left(\sum_{I\in{[n]\choose k}} S_{\mu+(r^I)+(1^J)}W\right). 
\end{equation}
Consider the collection of partitions $\mc{P}=\{\a:\a/\mu\rm{ is a vertical strip of size }l\}$, so that
\[S_{\mu}W\oo\left(\bw^l W\right)=\sum_{\a\in\mc{P}} S_{\a}W,\]
and note that $\ll\in\mc{P}$ and that $\a\prec\ll$ for every $\ll\neq\a\in\mc{P}$. We can then rewrite the left hand side of (\ref{eq:pkrSmubw^l}) as
\[\sum_{\a\in\mc{P}}p_{k,r}(W)\oo S_{\a}W,\]
so in order to prove (\ref{eq:genlPieri}) for $\ll$ it is sufficient to show that the right hand side of (\ref{eq:pkrSmubw^l}) is equal to
\[\sum_{\a\in\mc{P}}\left(\sum_{I\in{[n]\choose k}} S_{\a+(r^I)}W\right).\]
Since $\mc{P}=\{\mu+(1^J):J\in{[n]\choose l},\mu+(1^J)\rm{ dominant}\}$, we only have to check that when $\a=\mu+(1^J)$ is not dominant then
\begin{equation}\label{eq:sumSar^I=0}
\sum_{I\in{[n]\choose k}} S_{\a+(r^I)}W=0. 
\end{equation}
Note that the only way in which $\a=\mu+(1^J)$ can fail to be dominant is if for some index $j$, $\mu_j=\mu_{j+1}$ and $j\notin J$, $(j+1)\in J$. Fix such an index $j$, and note that $\a_{j+1}=\a_j+1$. It follows from (\ref{eq:vanishSll}) that when $I\subset[n]$ is such that both $j,j+1\in I$, or both $j,j+1\notin I$, then $S_{\a+(r^I)}W=0$. To show (\ref{eq:sumSar^I=0}) it is then enough to prove that
\begin{equation}\label{eq:sumjxorj+1=0}
\sum_{\substack{I\in{[n]\choose k} \\ j\in I,(j+1)\notin I}} S_{\a+(r^I)}W + \sum_{\substack{I'\in{[n]\choose k} \\ j\notin I',(j+1)\in I'}} S_{\a+(r^{I'})}W = 0.
\end{equation}
There is a one-to-one correspondence between the collection of subsets $I$ with $j\in I$, $(j+1)\notin I$, and subsets $I'$ with $j\notin I'$, $(j+1)\in I'$, given by $I'=(I\cup\{j+1\})\setminus\{j\}$. Moreover, for such a pair $I,I'$ it follows from (\ref{eq:defSll}) that $S_{\a+(r^I)}W=-S_{\a+(r^{I'})}W$ (because $\a+(r^I)+\d$ is obtained from $\a+(r^{I'})+\d$ by switching the $j$-th part with the $(j+1)$-st part), which proves (\ref{eq:sumjxorj+1=0}) and concludes the proof of the lemma.
\end{proof}

\subsection{Bott's theorem for Grassmannians {\cite[Ch.~4]{weyman}}}\label{subsec:bott}

We consider $X=\bb{G}(k,V)$, the Grassmannian of $k$-dimensional quotients of $V$ (or $k$-dimensional subspaces of $W=V^*$), with the tautological sequence
\begin{equation}\label{eq:tautGr}
0\lra\mc{R}\lra V\oo\mc{O}_X\lra\mc{Q}\lra 0, 
\end{equation}
where $\mc{Q}$ is the tautological rank $k$ quotient bundle, and $\mc{R}$ is the tautological rank $(n-k)$ sub-bundle. Bott's Theorem for Grassmannians \cite[Corollary~4.1.9]{weyman} computes the cohomology of a large class of $\GL$-equivariant bundles on $X$. We only need a weaker version that computes Euler characteristics.

Suppose that $\mc{M}$ is a quasi-coherent $\GL(W)$-equivariant sheaf on $X$. We say that $\mc{M}$ has \defi{admissible (resp. finite) cohomology} if its cohomology groups $H^j(X,\mc{M})$ are admissible (resp. finite) for $j=0,\cdots,\dim(X)$. We can therefore make sense of the Euler characteristic of $\mc{M}$ as an element of $\G(W)$ (resp. $R(W)$). We define the \defi{Euler characteristic} of $\mc{M}$ to be the virtual representation
\begin{equation}\label{eq:defEulerChar}
 \chi(X,\mc{M}) = \sum_{j=0}^{k\cdot(n-k)}(-1)^j H^j(X,\mc{M}).
\end{equation}

\begin{theorem}[Bott]\label{thm:bott}
 Let $\a\in\bb{Z}^k_{\dom}$ and $\b\in\bb{Z}^{n-k}_{\dom}$ be dominant weights, and let $\ll=(\a,\b)\in\bb{Z}^n$ be their concatenation. The Euler characteristic of $S_{\a}\mc{Q}\oo S_{\b}\mc{R}$ is given (with the convention (\ref{eq:defSll})) by
\[\chi(X,S_{\a}\mc{Q}\oo S_{\b}\mc{R})=S_{\ll}V.\]
\end{theorem}

We can now give an alternative interpretation of the elements $p_{k,r}$ introduced in (\ref{eq:defpkr}):

\begin{lemma}\label{lem:pkrispfwd}
 If we let $\Omega_X^i=\bw^i(\mc{R}\oo\mc{Q}^*)$ denote the sheaf of $i$-differential forms on $X$, and write $\mc{O}_X(1)=\det(\mc{Q})$ for the Pl\"ucker line bundle on $X$, then
\[p_{k,r}(V)=\sum_{i=0}^{k\cdot(n-k)}(-1)^i\cdot\chi(X,\Omega_X^i(r)).\]
\end{lemma}

\begin{proof}
Cauchy's formula \cite[Cor.~2.3.3]{weyman} yields
\[\bw^i(\mc{R}\oo\mc{Q}^*)=\bigoplus_{\mu\in P(k,n-k),\ |\mu|=i} S_{\mu}\mc{Q}^*\oo S_{\mu'}\mc{R}.\]
Twisting by $\mc{O}_X(r)=\det(\mc{Q})^{\oo r}=S_{(r^k)}\mc{Q}$, and taking Euler characteristics, we get using Theorem~\ref{thm:bott}
\[\chi(X,\Omega^i_X(r))=\sum_{\mu\in P(k,n-k),\ |\mu|=i}S_{(r-\mu_k,r-\mu_{k-1},\cdots,r-\mu_1,\mu'_1,\cdots,\mu'_{n-k})}V.\]
Using (\ref{eq:Sll+rI=Sll(rI)}) with $\ll=0$, we get $S_{(r-\mu_k,r-\mu_{k-1},\cdots,r-\mu_1,\mu'_1,\cdots,\mu'_{n-k})}V=(-1)^{|\mu|}\cdot S_{(r^I)}V$, so
\[\sum_{i=0}^{k\cdot(n-k)}(-1)^i\cdot\chi(X,\Omega_X^i(r))=\sum_{I\in{[n]\choose k}}S_{(r^I)}V=p_{k,r}(V).\qedhere\]
\end{proof}

\subsection{$\D$-modules {\cite{borel}, \cite{hottaetal}}}\label{subsec:Dmods}

For a smooth algebraic variety $X$ over $\bb{C}$, we let $\D_X$ denote the sheaf of \defi{differential operators} on $X$ \cite[Section~1.1]{hottaetal}. A \defi{$\D$-module} $\mc{M}$ on $X$ (or a $\D_X$-module) is a quasi-coherent sheaf $\mc{M}$ on $X$, with a left module action of $\D_X$. 

\begin{definition}\label{def:GequivariantDmod}
 Let $G$ be an algebraic group acting on $X$, and let $\mc{M}$ be a $\D_X$-module. Differentiating the action of $G$ on $X$ yields a map $d:\Lie(G)\to\Der_X$ from the Lie algebra of $G$ to the vector fields on $X$. The $\D_X$-module operation
 \begin{equation}\label{eq:Dmodaction}
\D_X\oo\mc{M}\to\mc{M},  
 \end{equation}
 composed with $d$ yields an action of $\Lie(G)$ on $\mc{M}$. The $\D_X$-module $\mc{M}$ is \defi{$G$-equivariant} if 
 \begin{enumerate}
  \item[(a)] $\mc{M}$ admits an action of $G$ compatible with (\ref{eq:Dmodaction}) (see \cite[Def.~11.5.2]{hottaetal} for a precise meaning of compatibility). 
  \item[(b)] The action of $\Lie(G)$ on $\mc{M}$ obtained by differentiating the action of $G$ on $\mc{M}$ coincides with the one induced from $d:\Lie(G)\to\Der_X$ and (\ref{eq:Dmodaction}).
 \end{enumerate}
\end{definition}

As discussed in the Introduction, examples of $G$-equivariant holonomic $\D_X$-modules are $\mc{O}_X$, and for a $G$-invariant subset $Y\subset X$, the local cohomology modules $\mc{H}_Y^{\bullet}(X,\mc{O}_X)$, as well as the intersection homology $\D$-modules $\mc{L}(Y,X)$. When $X=U$ is a vector space, and $Y=\{0\}$ is the origin, we let
\begin{equation}\label{eq:defE}
 E = \mc{L}(\{0\},U) = \mc{H}_{\{0\}}^{\dim(U)}(U,\mc{O}_U),
\end{equation}
be the unique simple $\D_U$-module supported at the origin. As a vector space (and a $G$-representation)
\begin{equation}\label{eq:E=detSym}
E=\det(U)\oo\Sym(U). 
\end{equation}

The following theorem gives a classification of the simple equivariant holonomic $\D$-modules, for a group action with finitely many orbits (see \cite[Section~11.6]{hottaetal}):

\begin{theorem}\label{thm:equivRH}
 Let $G$ be an algebraic group acting with finitely many orbits on a smooth algebraic variety~$X$. There is a one-to-one correspondence between:
 \begin{enumerate}
  \item[(a)] Simple $G$-equivariant holonomic $\D_X$-modules.
  \item[(b)] Pairs $(O,L)$ where $O$ is a $G$-orbit, and $L$ is an irreducible $G$-equivariant local system on $O$.
  \item[(b')] Pairs $(O,L)$ where $O$ is a $G$-orbit, and $L$ is an irreducible representation of the component group of the isotropy group of $O$.
 \end{enumerate}
\end{theorem}
\noindent Here by the \defi{isotropy group} of $O$ we mean the stabilizer of any element in $O$ (they are all isomorphic). For an algebraic group $H$, we denote by $H_0$ the connected component of the identity, which is a normal subgroup of $H$. The quotient $H/H_0$ is called the \defi{component group} of $H$.

\begin{remark}\label{rem:Ltrivial}
 When the representation $L$ in Theorem~\ref{thm:equivRH}(b') is trivial, the corresponding $\D_X$-module in part (a) is $\mc{L}(\ol{O},X)$, where $\ol{O}$ is the closure of $O$. It follows that in the case when the isotropy groups for the $G$-action on $X$ are connected, there is a one-to-one correspondence between simple $G$-equivariant $\D_X$-modules and orbits of the group action.
\end{remark}

Let $m\geq n$ be positive integers and consider the complex vector spaces $M$ of general $m\times n$ matrices, $M^{\symm}$ of $n\times n$ symmetric, and $M^{\skew}$ of $n\times n$ skew-symmetric matrices respectively. These spaces admit a natural action of a group $\GL$ via row and column operations: $\GL_m(\bb{C})\times\GL_n(\bb{C})$ acts on $M$, and $\GL_n(\bb{C})$ acts $M^{\symm}$ and $M^{\skew}$. We write $M_s$ (resp. $M^{\symm}_s$) for the subvariety of $M$ (resp. $M^{\symm}$) consisting of matrices of rank at most $s$, for $s=0,\cdots,n$, and $M^{\skew}_s$ for the subvariety of $M^{\skew}$ consisting of skew-symmetric matrices of rank at most $2s$, for $s=0,\cdots,\lfloor n/2\rfloor$. We have the following:

\begin{theorem}[Classification of simple $\GL$-equivariant holonomic $\D$-modules on spaces of matrices]\label{thm:classificationDmods}\

\begin{itemize}
 \item (General matrices). There are $(n+1)$ simple $\GL$-equivariant $\D$-modules on the vector space $M$ of $m\times n$ matrices, namely the intersection homology $\D$-modules $\mc{L}(M_{s},M)$, $s=0,\cdots,n$.
 \item (Symmetric matrices). There are $(2n+1)$ simple $\GL$-equivariant $\D$-modules on the vector space $M^{\symm}$ of $n\times n$ symmetric matrices, $(n+1)$ of which are the intersection homology $\D$-modules $\mc{L}(M^{\symm}_{s},M^{\symm})$, $s=0,\cdots,n$, while the remaining ones are the intersection homology $\D$-modules $\mc{L}(M^{\symm}_{s},M^{\symm};1/2)$, $s=1,\cdots,n$, corresponding to the non-trivial irreducible equivariant local systems on the orbits.
 \item (Skew-symmetric matrices). There are $(\lfloor n/2\rfloor+1)$ simple $\GL$-equivariant $\D$-modules on the vector space $M^{\skew}$ of $n\times n$ skew-symmetric matrices, namely $\mc{L}(M^{\skew}_{s},M^{\skew})$, $s=0,\cdots,\lfloor n/2\rfloor$.
\end{itemize}

\end{theorem}

\begin{proof}
 The theorem follows from Theorem~\ref{thm:equivRH} and Remark~\ref{rem:Ltrivial}, since the isotropy groups for general and skew-symmetric matrices are connected, while for symmetric matrices the isotropy groups of the non-zero orbits have two connected components.
\end{proof}

\subsection{Computing Euler characteristics}\label{subsec:pfwdDmods}

Let $X$ be a smooth complex projective algebraic variety and denote its dimension by $d_X$. Consider a finite dimensional vector space $U$, and a short exact sequence
\begin{equation}\label{eq:basicses}
0\lra\xi\lra U\oo\mc{O}_X\lra\eta\lra 0, 
\end{equation}
where $\xi,\eta$ are locally free sheaves on $X$. We think of $U^*$ as an affine space, and of~$U$ as linear forms on~$U^*$. We let $Y=\rm{Tot}_X(\eta^*)$ denote the total space of the bundle $\eta^*$, and define a morphism $\pi:Y\to U^*$ via the commutative diagram
\begin{equation}\label{eq:diagrGeneric}
\xymatrix{
Y=\rm{Tot}_X(\eta^*) \ar@{^{(}->}[r] \ar[dr]_{\pi} & U^*\times X \ar[d] \\
 & U^* \\
}
\end{equation}
where the top map is the inclusion of $\eta^*$ into the trivial bundle~$U^*$, and the vertical map is the projection onto the $U^*$ factor. We will be interested in understanding the (Euler characteristic of the) $\D$-module pushforward $\int_{\pi}\mc{M}$ along the map $\pi$ for certain $\D_Y$-modules $\mc{M}$. For affine morphisms $X'\to X$, we will identify freely quasi-coherent sheaves on $X'$ with quasi-coherent $\mc{O}_{X'}$-modules on $X$ as in \cite[Exercise~II.5.17(e)]{hartshorne}.

We let $\mc{S}=\Sym_{\mc{O}_X}(\eta)$, so that $Y=\ul{\Spec}_X(\mc{S})$, and consider a locally free sheaf $\mc{L}$ of rank one with $\mc{L}\subset\Sym^i(\eta)$ for some $i>0$. We pull-back $\mc{L}$ to $Y$, define $L=\rm{Tot}_Y(\mc{L}^*)$ to be the total space of the line bundle $\mc{L}^*$, and write $p:L\to Y$ for the natural map. The inclusion $\mc{L}\subset\Sym^i(\eta)$ defines a section $z:Y\to L$ of $p$ \cite[Exercise~II.5.18(c)]{hartshorne}, and we define $Z$ to be the zero-locus of $z$. If $X=\rm{Spec}(\bb{C})$ then $Y$ is an affine space, $\mc{S}$ is the ring of polynomial functions on $Y$, $\mc{L}$ corresponds to (the vector space spanned by) a polynomial $f\in\mc{S}$ of degree $i$, and $Z$ is the vanishing locus of $f$.

We consider the complement $Y^0=Y\setminus Z$ and let $j:Y^0\to Y$ denote the inclusion. Since $j$ is an affine open immersion, $\int_j\mc{O}_{Y^0}=\mc{O}_{Y^0}$ can be thought of as a quasi-coherent sheaf of algebras on $Y$ (or on $X$):
\[\mc{O}_{Y^0}=\varinjlim_r \mc{L}^{-r}\oo\mc{O}_Y=\varinjlim_r \mc{L}^{-r}\oo\mc{S}.\]
In the case when $X=\rm{Spec}(\bb{C})$, we have $\mc{O}_{Y^0}=\mc{S}_f$ is the localization of $\mc{S}$ at $f$, which is a $\D$-module on the affine space $Y$. We define the quasi-coherent sheaf $\mc{S}^{\vee}$ on $X$ (the \defi{graded dual of $\mc{S}$}) by
\[\mc{S}^{\vee} = \det(\eta^*)\oo\Sym_{\mc{O}_X}(\eta^*).\]

\begin{proposition}\label{prop:EuleraslimPr}
 With the notation above, we assume that $X$ admits an action of a reductive group $G$, that $U$ is a finite dimensional $G$-representation, and that $\xi,\eta,\mc{L}$ are $G$-equivariant locally free sheaves. Assume further that we have an isomorphism of $G$-equivariant quasi-coherent sheaves on $X$
\begin{equation}\label{eq:twolims}
\mc{O}_{Y^0}\simeq\varinjlim_r\mc{L}^{r}\oo\mc{S}^{\vee}. 
\end{equation}
Let $\mc{M}$ be a $\D_Y$-module which is isomorphic, as a quasi-coherent $G$-equivariant sheaf on $X$, to $\mc{O}_{Y^0}\oo_{\mc{O}_X}\mc{L}'$, with $\mc{L}'$ a line bundle on $X$. We denote by $\Omega^i_X$ the sheaf of $i$-differential forms on $X$, and assume that for every $i=0,\cdots,d_X$ the sheaves $\Omega^i_X\oo\mc{M}\oo\det(\xi^*)\oo\Sym_{\mc{O}_X}(\xi^*)$ have $G$-admissible cohomology.
If we define the sequence $P_r(X,\mc{L};\mc{L}')\in\G(G)$ via 
\[P_r(X,\mc{L};\mc{L}')=\sum_{i=0}^{d_X} (-1)^{d_X-i}\cdot\chi(X,\mc{L}^r\oo\mc{L}'\oo\Omega^i_X),\]
then
\begin{equation}\label{eq:chiMlimitPr}
 \chi\left(\int_{\pi}\mc{M}\right)=\lim_{r\to\infty} P_r(X,\mc{L};\mc{L}')\oo\det(U^*)\oo\Sym_{\bb{C}}(U^*).
\end{equation}
\end{proposition}

\begin{remark}\label{rem:Euleraslim}
 We will apply Proposition~\ref{prop:EuleraslimPr} in the case when $X=\bb{G}(k,V)$ is a Grassmann variety, and $\mc{L}=\mc{O}_X(1)$ is the Pl\"ucker line bundle (or its square). It follows from Lemma~\ref{lem:pkrispfwd} that 
 \[P_r(X,\mc{O}_X(1);\mc{O}_X)=(-1)^{k(n-k)}\cdot p_{k,r}(V).\]
 It follows that if $X=\bb{G}(k,V_1)\times\bb{G}(k,V_2)$, where $\dim(V_1)=m$, $\dim(V_2)=n$, and if $\mc{L}=\mc{O}_X(1,1)$ then $P_r(X,\mc{O}_X(1);\mc{O}_X)=(-1)^{k\cdot(m-n)}\cdot p_{k,r}(V_1)\oo p_{k,r}(V_2)$.
\end{remark}

\begin{proof}[Proof of Proposition~\ref{prop:EuleraslimPr}]
 Since the sheaves $\Omega^i_X\oo\mc{M}\oo\det(\xi^*)\oo\Sym_{\mc{O}_X}(\xi^*)$ have admissible cohomology, it follows from \cite[Corollary~2.10]{raicu-VeroDmods} that
\begin{equation}\label{eq:intM}
\chi\left(\int_{\pi}\mc{M}\right) = \sum_{i=0}^{d_X} (-1)^{d_X-i}\cdot\chi(X,\Omega^i_X\oo\mc{M}\oo\det(\xi^*)\oo\Sym_{\mc{O}_X}(\xi^*)). 
\end{equation}
Computing Euler characteristics commutes with colimits and associated graded constructions. By (\ref{eq:basicses}) we get a filtration of $U^*\oo\mc{O}_X$ with $\gr(U^*\oo\mc{O}_X)=\xi^*\oplus\eta^*$, which yields a filtration of $\Sym_{\mc{O}_X}(U^*\oo\mc{O}_X)$ with $\gr(\Sym_{\mc{O}_X}(U^*\oo\mc{O}_X))=\Sym_{\mc{O}_X}(\xi^*)\oo\Sym_{\mc{O}_X}(\eta^*)$. We also get that $\det(U^*\oo\mc{O}_X)=\det(\xi^*)\oo\det(\eta^*)$, and therefore
\[
\begin{aligned}
\chi(X,\Omega_X^i\oo\mc{L}'\oo\mc{L}^r\oo\mc{S}^{\vee}\oo\det(\xi^*)&\oo\Sym_{\mc{O}_X}(\xi^*)) = \chi(X,\Omega_X^i\oo\mc{L}'\oo\mc{L}^r\oo\det(U^*)\oo\Sym_{\mc{O}_X}(U^*\oo\mc{O}_X)) \\
&\overset{U^*\rm{ is a trivial bundle}}{=} \chi(X,\Omega_X^i\oo\mc{L}'\oo\mc{L}^r)\oo\det(U^*)\oo\Sym_{\bb{C}}(U^*).
\end{aligned}
\]
Multiplying this equality by $(-1)^{d_X-i}$, summing over $i=0,\cdots,d_X$, taking the limit as $r\to\infty$, and using the identification (\ref{eq:twolims}) tensored with $\mc{L}'$, we get (\ref{eq:chiMlimitPr}).
\end{proof}

\subsection{The Weyl algebra and the Fourier transform}\label{subsec:Fourier}

For a positive integer $N$, the \defi{Weyl algebra} 
\begin{equation}\label{eq:defWeyl}
 \bb{C}[x_1,\cdots,x_N,\partial_1,\cdots,\partial_N],\ \partial_i=\frac{\partial}{\partial x_i} 
\end{equation}
is the ring of differential operators on $\bb{C}^N$. In this section we give a coordinate independent description of the Weyl algebra, and use it to describe the Fourier transform.

Given a finite dimensional $\bb{C}$--vector space $U$ of dimension $N$, we write $\scpr{}{}$ for the natural pairing $U\times U^*\to\bb{C}$. We let $\tl{U}=U\oplus U^*$ and define a non--degenerate skew--symmetric form $\om:\tl{U}\oo\tl{U}\to\bb{C}$ by
\[\om(u,u')=
\begin{cases}
 \scpr{u}{u'} & \textrm{if }u\in U,\ u'\in U^*, \\
 -\scpr{u'}{u} & \textrm{if }u'\in U,\ u\in U^*, \\
 0 & \textrm{otherwise}. \\
\end{cases}
\]
We write $T_n(\tl{U})$ for the tensor product $\tl{U}^{\oo n}$, and let $T(\tl{U})=\bigoplus_{n\geq 0}T_n(\tl{U})$ denote the \defi{tensor algebra} on $\tl{U}$. We have a natural inclusion $\bw^2\tl{U}\subset T_2(\tl{U})$, and define the \defi{Weyl algebra} $\mc{D}_{U^*}$ as the quotient
\begin{equation}\label{eq:WeylAlgebra}
\mc{D}_{U^*}=T(\tl{U})/\langle x-\om(x):x\in\bw^2\tl{U}\rangle
\end{equation}
of the tensor algebra by the bilateral ideal generated by differences $x-\om(x)$, with $x\in\bw^2\tl{U}$. Note that $\D_{U^*}$ is the ring of differential operators on the vector space $U^*$. If we choose a basis $x_1,\cdots,x_N$ of $U$, and the dual basis $\partial_1,\cdots,\partial_N$ of $U^*$, then $\D_{U^*}$ coincides with (\ref{eq:defWeyl}).

\begin{lemma}[Fourier transform]\label{lem:Fourier}
 If $M$ is a (left) $\mc{D}_{U}$-module, then $\det(U^*)\oo M$ has the structure of a (left) $\mc{D}_{U^*}$-module.
\end{lemma}

\begin{example}
 The most basic example is when $M=\Sym(U^*)$ is the coordinate ring of $U$. In that case $\det(U^*)\oo\Sym(U^*)$ is equal to $E$, the simple holonomic $\D_{U^*}$-module supported at the origin (see (\ref{eq:defE})).
\end{example}

\begin{proof}[Proof of Lemma~\ref{lem:Fourier}]
 Using the identification of $\tl{U}^*$ with $\tl{U}$ coming from the natural isomorphism $U^*\oplus U\simeq U\oplus U^*$, it is easy to see that $\mc{D}_{U^*}\simeq\mc{D}_{U}^{\op}$, where \textsuperscript{op} denotes the opposite ring. Since $M$ is a left $\mc{D}_U$-module, it is also a right $\mc{D}_U^{\op}$-module, i.e. it can be identified with a right $\mc{D}_{U^*}$-module. The canonical sheaf $\om_{U^*}$ on the vector space $U^*$ is a free rank one module generated by $\det(U)$. By \cite[Prop.~1.2.12]{hottaetal}, the association $M\mapsto\om_{U^*}^{-1}\oo M=\det(U^*)\oo M$ gives an equivalence between the categories of right $\mc{D}_{U^*}$-modules and left $\mc{D}_{U^*}$-modules.
\end{proof}

Motivated by Lemma~\ref{lem:Fourier}, we define a \defi{Fourier transform relative to $U$}, denoted $\mc{F}_U$, on the Grothendieck group $\G(G)$ of admissible $G$-representations as follows:
\begin{equation}\label{eq:defGFourierU}
\mc{F}_U\left(\sum a_i\cdot M_i\right)=\sum a_i\cdot(\det(U^*)\oo M_i^*).
\end{equation}
The context in which we apply the Fourier transform is as follows: we will have constructions which are functorial in $U$ for certain $\mc{D}_U$-modules $M_U$ which are admissible representations for some group $G$, in such a way that
\[M_U=\bigoplus_i M_i^{\oplus a_i}\rm{ if and only if }M_{U^*}=\bigoplus_i (M^*_i)^{\oplus a_i}.\]
By Lemma~\ref{lem:Fourier}, the Fourier transform of the $\mc{D}_{U^*}$-module $M_{U^*}$ has character equal to $\mc{F}_U(\sum_i a_i\cdot M_i)$. We will slightly imprecisely refer to this as the character of the Fourier transform of $M_U$.

\subsection{A little linear algebra}\label{subsec:linalg}

Consider a finite partially ordered set $\mc{P}$, and let $\mc{A}$ denote the free abelian group with basis $\{\mf{v}_p:p\in\mc{P}\}$. We write $p\succ q$ to indicate that $p$ is strictly larger than $q$ with respect to the partial order, and $p\succeq q$ when we allow equality. Assume that $\mc{F}:\mc{P}\lra\mc{P}$ is an order reversing bijection, i.e. $p\succeq q$ if and only if $\mc{F}(q)\succeq\mc{F}(p)$. By abuse of notation, we also write $\mc{F}:\mc{A}\lra\mc{A}$ for the induced automorphism of $\mc{A}$, given by $\mc{F}(\mf{v}_p)=\mf{v}_{\mc{F}(p)}$. We have the following:

\begin{lemma}\label{lem:uppertriangularFourier}
 Suppose that we have a collection of elements $v_p\in\mc{A}$ for $p\in\mc{P}$, for which there exist relations
 \begin{equation}\label{eq:vtomfv}
  v_p=\mf{v}_p+\sum_{q\succ p} a^p_q\cdot\mf{v}_q,\rm{ for some integers }a^p_q.  
 \end{equation}
 If the automorphism $\mc{F}$ of $\mc{A}$ permutes the elements $v_p$ then $v_p=\mf{v}_p$ for all $p\in\mc{P}$ (and hence all $a^p_q=0$).
\end{lemma}

\begin{proof}
 Write $\mc{F}(v_p)=v_{\s(p)}$ for some permutation $\s:\mc{P}\lra\mc{P}$. Applying $\mc{F}$ to (\ref{eq:vtomfv}) we get
 \[v_{\s(p)}=\mf{v}_{\mc{F}(p)}+\sum_{q\succ p} a^p_q\cdot\mf{v}_{\mc{F}(q)},\]
 which is necessarily a permutation of the relations (\ref{eq:vtomfv}). Since $\mc{F}$ is order-reversing, it follows that $\s(p)=\mc{F}(q)$ for some $q\succeq p$, and if $\s(p)=\mc{F}(p)$ then one also has $v_{\s(p)}=\mf{v}_{\mc{F}(p)}$, i.e. $a^p_q=0$ for all $q\succ p$. We get that $\mc{F}(p)\succeq\s(p)$ for all $p\in\mc{P}$, and the equality $\mc{F}(p)=\s(p)$ implies $v_{\s(p)}=\mf{v}_{\mc{F}(p)}$. An easy induction on the \defi{height} of $\mc{F}(p)$, defined by $\htt(\mc{F}(p))=\#\{q:\mc{F}(p)\succ q\}$, shows that $\mc{F}(p)=\s(p)$ for all $p$, which concludes the proof of the lemma.
\end{proof}

\section{Some limit calculations in the Grothendieck group of admissible representations}\label{sec:lims}

Recall the terminology from Sections~\ref{subsubsec:admissible}--\ref{subsubsec:genlPieri} which we will be using freely throughout this section. In particular recall the notation $\G(G)$ for the Grothendieck group of admissible $G$-representations for some group $G$, and the definition of $p_{k,r}(V)$ from (\ref{eq:defpkr}) (also Lemma~\ref{lem:pkrispfwd}). When $W$ is a vector space, we write $V=W^*$ for its dual. In this section we compute in three cases limits in $\G(G)$ of the type
\begin{equation}\label{eq:limspkrE}
\lim_{r\to\infty} p_{k,r}\oo E, 
\end{equation}
where $(p_{k,r})_{r}$ is a sequence of finite virtual $G$-representations, $E=\det(U)\oo\Sym(U)$ is the (character of the) simple $\D_U$-module supported at the origin (\ref{eq:defE}), where $U$ is a finite dimensional $G$-representation:
\begin{itemize}
 \item $U=\Sym^2 W$, $G=\GL(W)$ (so that $\G(G)=\G(W)$), $p_{k,r}=p_{k,r}(V)$. The limit (\ref{eq:limspkrE}) does not exist if $r$ is arbitrary, but instead we have to consider the cases when $r$ is even resp. odd separately.
 \item $U=W_1\oo W_2$, $G=\GL(W_1)\times\GL(W_2)$ (so that $\G(G)=\G(W_1,W_2)$), $p_{k,r}=p_{k,r}(V_1)\oo p_{k,r}(V_2)$.
 \item $U=\bw^2 W$, $G=\GL(W)$, $p_{k,r}=p_{k,r}(V)$ with $k$ even.
\end{itemize}
As mentioned in the Introduction and explained in Section~\ref{subsec:pfwdDmods}, the limits (\ref{eq:limspkrE}) correspond to Euler characteristic calculations for certain $\D$-module direct images. They are essential to the character calculations in Sections~\ref{sec:symm}--\ref{sec:skew} below. The reader who is not interested in the details of the limit calculations may wish to record the results of Propositions~\ref{prop:limsym},~\ref{prop:limsgeneral}, and~\ref{prop:limskew} below, and skip to Section~\ref{sec:symm}.

\subsection{Symmetric matrices}\label{subsec:limsym}

We let $W$ be a vector space of dimension $n$. For $s=0,\cdots,n$ and $j=1,2,$ we define the elements $\mf{C}^j_s\in\G(W)$ via
\begin{equation}\label{eq:defcjs}
\mf{C}^j_s=\bigoplus_{\ll\in\mc{C}^j(s,n)} S_{\ll}W. 
\end{equation}
where $\mc{C}^j(s,n)$ is defined in (\ref{eq:defmcC^isn}).

\begin{proposition}\label{prop:limsym}
If $E=\det\left(\Sym^2 W\right)\oo\Sym\left(\Sym^2 W\right)$ then for $k=0,\cdots,n$,
\[
 (-1)^{k(n-k)}\cdot\left(\lim_{\substack{r\to\infty \\ r\equiv k+1\ (\opmod\ 2)}} p_{k,r}(V)\oo E\right)=\begin{cases}
 \displaystyle\sum_{\substack{s=n-k \\ s\rm{ even}}}^n {\frac{s-2}{2}\choose\frac{n-k-2}{2}}\cdot \mf{C}_s^2 + \sum_{\substack{s=n-k+1 \\ s\rm{ odd}}}^n {\frac{s-1}{2}\choose\frac{n-k}{2}}\cdot \mf{C}_s^1 & \rm{if }n-k\rm{ even,} \\
 & \\
 \displaystyle\sum_{\substack{s=n-k \\ s\rm{ odd}}}^n {\frac{s-1}{2}\choose\frac{n-k-1}{2}}\cdot \mf{C}_s^1 - \sum_{\substack{s=n-k+1 \\ s\rm{ even}}}^n {\frac{s-2}{2}\choose\frac{n-k-1}{2}}\cdot \mf{C}_s^2 & \rm{if }n-k\rm{ odd.}
\end{cases}
\]
\[
 (-1)^{k(n-k)}\cdot\left(\lim_{\substack{r\to\infty \\ r\equiv k\ (\opmod\ 2)}} p_{k,r}(V)\oo E\right)=\begin{cases}
 \displaystyle\sum_{\substack{s=n-k \\ s\rm{ even}}}^n {\frac{s}{2}\choose\frac{n-k}{2}}\cdot \mf{C}_s^1 & \rm{if }n-k\rm{ even,} \\
 & \\
 \displaystyle\sum_{\substack{s=n-k \\ s\rm{ odd}}}^n {\frac{s-1}{2}\choose\frac{n-k-1}{2}}\cdot \mf{C}_s^2 & \rm{if }n-k\rm{ odd.}
\end{cases}
\]
\end{proposition}

When $k=0$ the above equalities are easy to verify: $p_{k,r}(V)=\bb{C}$ is the trivial representation, so the left hand side reduces to $E$, regardless of the parity of $r$; the right hand side is either $\mf{C}_n^1$ or $\mf{C}_n^2$, but $E=\mf{C}_n^1=\mf{C}_n^2$. We therefore fix $1\leq k\leq n$ for the rest of this section. We begin with some notation and preliminary results before proving the proposition. For $j\in\bb{Z}/2\bb{Z}$ we let
\begin{equation}\label{eq:defmcCjk}
\begin{aligned}
 \mc{C}^j &=\{\ll\in\bb{Z}^{k}_{\dom}:\ll_{i}\equiv n+1+j\ (\opmod\ 2)\rm{ for }i=1,\cdots,k\}, \\
 \mc{C}^j_{\geq n+1} &=\{\ll\in\bb{Z}^{n-k}_{\dom}:\ll_{i}\equiv n+1+j\ (\opmod\ 2)\rm{ for }i=1,\cdots,n-k,\rm{ and }\ll_{n-k}\geq n+1\}.
\end{aligned}
\end{equation}

With the convention $\ll_0=\infty$, $\ll_{n+1}=-\infty$, we define for $s=0,\cdots,n$, 
\begin{equation}\label{eq:ZsnC}
 \mc{Z}(s)=\{\ll\in\bb{Z}^n_{\dom}:\ll_s\geq s+1\geq\ll_{s+1}\},
\end{equation}
and note that the sets $\mc{Z}(s)$, $s=0,\cdots,n$ form a partition of $\bb{Z}^n_{\dom}$. For $h,j\in\bb{Z}/2\bb{Z}$ we let

\begin{equation}\label{eq:defChj}
 \mc{C}^{h,j}(s)=\left\{\ll\in\mc{Z}(s):\ll_i\overset{(\opmod\ 2)}{\equiv}\begin{cases}
  h, & \rm{for }i=1,\cdots,s, \\
  j, & \rm{for }i=s+1,\cdots,n.
\end{cases}
\right\},\quad
\mf{C}^{h,j}_s=\sum_{\ll\in\mc{C}^{h,j}(s)}S_{\ll}W.
\end{equation}
Comparing with (\ref{eq:defmcC^isn}) we get that $\mc{C}^1(s,n)=\mc{C}^{s+1,s+1}(s)\cup\mc{C}^{s+1,s+1}(s+1)$ and $\mc{C}^2(s,n)=\mc{C}^{s+1,s}(s)$ so
\begin{equation}\label{eq:ChjtoCj}
 \mf{C}^1_s=\mf{C}^{s+1,s+1}_s+\mf{C}^{s+1,s+1}_{s+1},\quad\rm{and}\quad\mf{C}^2_s=\mf{C}^{s+1,s}_s.
\end{equation}

\begin{lemma}\label{lem:ll^iinmcC}
 If $I\in{[n]\choose k}$, $\ll^1(I)\in\mc{C}^{k+1+j}$ and $\ll^2(I)\in\mc{C}^0_{\geq n+1}$ then
 \begin{itemize}
  \item $\ll\in\mc{Z}(s)$ for some $s=n-k,\cdots,n$.
  \item $\{s+1,\cdots,n\}\subset I$.
  \item $\ll_{s+1}\equiv\cdots\equiv\ll_n\equiv j\ (\opmod\ 2)$.
 \end{itemize}
\end{lemma}

\begin{proof} Consider the unique $s$ for which $\ll\in\mc{Z}(s)$. Let $s'$ be the maximal element of $I^c$ and assume that $s'>s$. We have (using (\ref{eq:listIIc})) that $i_n=s'$ and therefore $\ll_{s'}\leq\ll_{s+1}\leq s+1$ and
\[\ll^2(I)_{n-k}\overset{(\ref{eq:defll12})}{=}n+\ll_{i_n}-i_n=n+\ll_{s'}-s'\leq n+s+1-s'<n+1,\]
which contradicts $\ll^2(I)\in\mc{C}^0_{\geq n+1}$. It follows that $s'\leq s$ and hence $\{s+1,\cdots,n\}\subset I$, which implies $n-s\leq k$, or $s\geq n-k$. From (\ref{eq:listIIc}) we get
\[i_t=t+n-k \rm{ for }t=k-n+s+1,\cdots,k,\]
which using the fact that $\ll^1(I)\in\mc{C}^{k+1+j}$ yields for $t=k-n+s+1,\cdots,k$
\[n+1+k+1+j\overset{(\opmod\ 2)}{\equiv}\ll^1(I)_t=t+\ll_{i_t}-i_t=t+\ll_{t+n-k}-(t+n-k)=\ll_{t+n-k}+k-n\]
so $\ll_{t+n-k}\equiv j\ (\opmod\ 2)$, concluding the proof of the lemma. 
\end{proof}

\begin{lemma}\label{lem:lemPllkj}
 Assume that $\ll\in\mc{Z}(s)$ and that there exists an index $1\leq i<s$ such that $\ll_i\not\equiv\ll_{i+1}\ (\opmod\ 2)$. For any $j\in\bb{Z}/2\bb{Z}$, consider the collection
\begin{equation}\label{eq:defmcPllkj}
\mc{P}_{\ll}(j)=\left\{I\in{[n]\choose k}:\ll^1(I)\in\mc{C}^{k+1+j},\ll^2(I)\in\mc{C}^0_{\geq n+1}\right\}. 
\end{equation}
We have (using notation (\ref{eq:defsI}))
\begin{equation}\label{eq:sumsgnsI=0}
 \sum_{I\in\mc{P}_{\ll}(j)}\sgn(\s(I))=0.
\end{equation}
\end{lemma}

\begin{proof} 
We show that if $I\in\mc{P}_{\ll}(j)$ then exactly one of $i,i+1$ is contained in $I$. Moreover, we show that the assignment $I'=I\setminus\{i\}\cup\{i+1\}$ establishes a bijection between
\begin{equation}\label{eq:corrII'}
\{I\in\mc{P}_{\ll}(j):i\in I\}\quad\rm{and}\quad\{I'\in\mc{P}_{\ll}(j):i+1\in I'\}. 
\end{equation}
Since $\sgn(\s(I'))=-\sgn(\s(I))$, the conclusion (\ref{eq:sumsgnsI=0}) follows.

Assume that $I$ is such that $i,i+1$ are both in $I$, or both in $I^c$. We can then find $t<k$ or $t>k$ such that $i_t=i$ and $i_{t+1}=i+1$. If $t<k$ then $\ll^1(I)_t\not\equiv\ll^1(I)_{t+1}\ (\opmod\ 2)$, contradicting $\ll^1(I)\in\mc{C}^{k+1+j}$. If $t>k$ then $\ll^2(I)_{t-k}\not\equiv\ll^2(I)_{t-k+1}\ (\opmod\ 2)$, contradicting $\ll^2(I)\in\mc{C}^0_{\geq n+1}$.

Choose now a set $I$ with $i\in I$, $i+1\in I^c$, and choose $t_0\leq k$, $t_1\geq k+1$, such that $i_{t_0}=i$, $i_{t_1}=i+1$. If we let $I'=I\setminus\{i\}\cup\{i+1\}$ then $\ll^1(I)_t=\ll^1(I')_t$ for $t\neq t_0$, and $\ll^2(I)_t=\ll^2(I')_t$ for $t\neq t_1-k$. We have
\[\ll^1(I)_{t_0}=t_0+\ll_i-i,\quad \ll^1(I')_{t_0}=t_0+\ll_{i+1}-(i+1),\]
\[\ll^2(I)_{t_1-k}=t_1+\ll_{i+1}-(i+1),\quad \ll^2(I')_{t_1-k}=t_1+\ll_i-i,\]
and since $\ll_i\not\equiv\ll_{i+1}\ (\opmod\ 2)$, we get $\ll^1(I)_t\equiv\ll^1(I')_t\ (\opmod\ 2)$ and $\ll^2(I)_t\equiv\ll^2(I')_t\ (\opmod\ 2)$ for all $t$. Since $\ll^2(I)_{t_1-k}\leq\ll^2(I')_{t_1-k}$, the only way in which the correspondence $I\leftrightarrow I'$ could fail to induce a bijection (\ref{eq:corrII'}) is if for some $I,I'$ we get $t_1=n$ and $\ll^2(I)_{n-k}\leq n<n+1\leq\ll^2(I')_{n-k}$, in which case $\ll^2(I)\not\in\mc{C}^0_{\geq n+1}$, but $\ll^2(I')\in\mc{C}^0_{\geq n+1}$. However, the inequality $\ll^2(I)_{n-k}\leq n$ would imply
\[\ll^2(I)_{t_1-k}=t_1+\ll_{i_{t_1}}-i_{t_1}=n+\ll_{i+1}-(i+1)\leq n\rm{ or equivalently }\ll_{i+1}\leq i+1.\]
Since $i<s$ by hypothesis, we get $\ll_s\leq\ll_{i+1}\leq(i+1)\leq s$, contradicting the fact that $\ll\in\mc{Z}(s)$.
\end{proof}

\begin{lemma}\label{lem:Plljtoparts}
 If $\ll\in\mc{C}^{h,j}(s)$, $s\geq n-k$, then there is a one-to-one correspondence between elements $\mc{P}_{\ll}(j)$ and the set $P^{n-k+j-h,s+1-h}(k-n+s,n-k)$ (defined in (\ref{eq:defPhjab})).
 Moreover, for every $I\in\mc{P}_{\ll}(j)$ we have
 \[\sgn(\s(I))=(-1)^{(n-k)\cdot(k+h)},\]
 and $\mc{P}_{\ll}(j)$ is empty if $h\equiv s\equiv j+1\ (\opmod\ 2)$.
\end{lemma}

\begin{proof} The correspondence between sets $I\in{[n]\choose k}$ (resp. their complements $I^c$) and partitions $\mu\in P(k,n-k)$ (resp. their conjugates $\mu'$) is given in (\ref{eq:Itomu}) (resp. (\ref{eq:Icfrommu'})). If $I\in\mc{P}_{\ll}(j)$ then it follows from Lemma~\ref{lem:ll^iinmcC} that $s+1,\cdots,n$ are the largest elements of $I$, namely $i_{k-n+s+1},\cdots,i_k$, so $\mu_1=\cdots=\mu_{n-s}=n-k$. The set $I$ is then determined by $\ol{\mu}=(\mu_{n-s+1},\cdots,\mu_k)\in P(k-n+s,n-k)$. Since $\ll\in\mc{C}^{h,j}(s)$, the condition $\ll^1(I)\in\mc{C}^{k+1+j}$ is equivalent to $\ol{\mu}_i\equiv n-k+j-h\ (\opmod\ 2)$. The condition $\ll^2(I)\in\mc{C}^0_{\geq n+1}$ is equivalent to $\mu'_i\equiv n+1-h\ (\opmod\ 2)$, which in turn is equivalent to $\ol{\mu}'_i\equiv s+1-h\ (\opmod\ 2)$. It follows that $I\in\mc{P}_{\ll}(j)$ if and only if $\ol{\mu}\in P^{n-k+j-h,s+1-h}(k-n+s,n-k)$, which establishes the desired bijection. Moreover
\[\sgn(\s(I))=(-1)^{|\mu|}=(-1)^{|\mu'|}=(-1)^{(n-k)\cdot(n+1-h)}=(-1)^{(n-k)\cdot(k+h)},\]
where the last equality follows from the fact that $(n-k)\cdot(n+1-k)$ is even. If $h\equiv s\equiv j+1\ (\opmod\ 2)$ then $|\mc{P}_{\ll}(j)|=|P^{n-k+1,1}(k-n+s,n-k)|=0$ by Lemma~\ref{lem:countpartitions}.
\end{proof}

\begin{proof}[Proof of Proposition~\ref{prop:limsym}]
 We have
\[\scpr{S_{\ll}W}{p_{k,r}(V)\oo E}=\scpr{S_{\ll}W\oo p_{k,r}(W)}{E}\overset{(\ref{eq:genlPieri}),(\ref{eq:Sll+rI=Sll(rI)})}{=}\sum_{I\in{[n]\choose k}}\sgn(\s(I))\cdot\scpr{S_{\ll(r,I)}W}{E}.\]
Since $\det(\Sym^2 W)=\det(W)^{\oo(n+1)}=S_{(n+1)^n}W$, we get using Cauchy's formula \cite[Prop.~2.3.8]{weyman} that
\[E=\det\left(\Sym^2 W\right)\oo\Sym\left(\Sym^2 W\right)=\bigoplus_{\substack{\ll\in\bb{Z}^n_{\dom},\ll_n\geq n+1 \\ \ll_i\equiv n+1\ (\opmod\ 2)}}S_{\ll}W.\]
Using notation (\ref{eq:defllrI}--\ref{eq:defll12}) and (\ref{eq:defmcCjk}) we obtain for $r\gg 0$
\[
 \scpr{S_{\ll(r,I)}W}{E} = \begin{cases}
 1, & \rm{if }\ll^1(I)\in\mc{C}^r\rm{ and }\ll^2(I)\in\mc{C}^0_{\geq n+1}, \\
 0, & \rm{otherwise}.
 \end{cases}
\]
It follows (using notation (\ref{eq:defmcPllkj})) that for $j\in\bb{Z}/2\bb{Z}$
\begin{equation}\label{eq:limpkrvsym1}
\lim_{\substack{r\to\infty \\ r\equiv k+1+j\ (\opmod\ 2)}} p_{k,r}(V)\oo E=\sum_{\ll\in\bb{Z}^n_{\dom}}\left(\sum_{I\in\mc{P}_{\ll}(j)}\sgn(\s(I))\right)\cdot S_{\ll}W, 
\end{equation}
and by Lemmas~\ref{lem:ll^iinmcC} and~\ref{lem:lemPllkj} we only need to consider $\ll\in\mc{Z}(s)$ for $s\geq n-k$ such that (for some $h\in\bb{Z}/2\bb{Z}$)
\[\ll_1\equiv\cdots\equiv\ll_s\equiv h\ (\opmod\ 2)\quad\rm{ and }\quad\ll_{s+1}\equiv\cdots\equiv\ll_n\equiv j\ (\opmod\ 2),\]
i.e. $\ll\in\mc{C}^{h,j}(s)$. Multiplying both sides of (\ref{eq:limpkrvsym1}) by $(-1)^{k\cdot(n-k)}$ and using Lemma~\ref{lem:Plljtoparts} we get
\[(-1)^{k\cdot(n-k)}\cdot\left(\lim_{\substack{r\to\infty \\ r\overset{(\opmod\ 2)}{\equiv} k+1+j}} p_{k,r}(V)\oo E\right)=\sum_{\substack{n-k\leq s\leq n \\ h=j,j+1}}(-1)^{(n-k)\cdot h}\cdot|P^{n-k+j-h,s+1-h}(k-n+s,n-k)|\cdot\mf{C}^{h,j}_s.\]
We separate the contributions of the right hand side according to two cases:

\ul{Terms with $h=j+1$:} By Lemma~\ref{lem:Plljtoparts} we can consider only the terms with $s\equiv j\ (\opmod\ 2)$, in which case we get from (\ref{eq:ChjtoCj}) that $\mf{C}^{h,j}_s=\mf{C}^{s+1,s}_s=\mf{C}^2_s$. We have
\[
 |P^{n-k+j-h,s+1-h}(k-n+s,n-k)|=|P^{n-k+1,0}(k-n+s,n-k)|\overset{\rm{Lemma }\ref{lem:countpartitions}}{=}\begin{cases}
\displaystyle{\lfloor\frac{s-1}{2}\rfloor \choose \frac{n-k-1}{2}} & n-k\rm{ odd}, \\
\\
\displaystyle{\frac{s-2}{2} \choose \frac{n-k-2}{2}} & n-k\rm{ and }s\rm{ even}, \\
\\
0 & \rm{otherwise.}
\end{cases}
\]
Comparing the coefficient of $\mf{C}^2_s$ in Proposition~\ref{prop:limsym} with $(-1)^{(n-k)\cdot h}\cdot|P^{n-k+j-h,s+1-h}(k-n+s,n-k)|$ in each of the cases $j=0,1$, and $(n-k)$ even and odd, we see that they agree.

\ul{Terms with $h=j$:} The terms with $s\equiv j+1\ (\opmod\ 2)$ contribute $\mf{C}^{h,j}_s=\mf{C}^{s+1,s+1}_s$ with coefficient $(-1)^{(n-k)\cdot h}\cdot|P^{n-k,0}(k-n+s,n-k)|$. The terms with $s\equiv j\ (\opmod\ 2)$ contribute $\mf{C}^{h,j}_s=\mf{C}^{s,s}_s$ with coefficient $(-1)^{(n-k)\cdot h}\cdot|P^{n-k,1}(k-n+s,n-k)|$. For $s=n-k$ we get $|P^{n-k,1}(k-n+s,n-k)|=0$ so $\mf{C}^{s,s}_s$ only appears for $s>n-k$. Observing that $|P^{n-k,0}(k-n+s,n-k)|=|P^{n-k,1}(k-n+s+1,n-k)|$ for $s\geq n-k$, and using $\mf{C}^1_s=\mf{C}^{s+1,s+1}_s+\mf{C}^{s+1,s+1}_{s+1}$ in (\ref{eq:ChjtoCj}), we conclude that the terms with $h=j$ contribute \[\sum_{s\equiv j+1\ (\opmod\ 2)}(-1)^{(n-k)\cdot h}\cdot|P^{n-k,0}(k-n+s,n-k)|\cdot\mf{C}^1_s,\]
where
\[|P^{n-k,0}(k-n+s,n-k)|\overset{\rm{Lemma }\ref{lem:countpartitions}}{=}\begin{cases}
\displaystyle{\lfloor\frac{s}{2}\rfloor \choose \frac{n-k}{2}} & n-k\rm{ even}, \\
\\
\displaystyle{\frac{s-1}{2} \choose \frac{n-k-1}{2}} & n-k\rm{ and }s\rm{ odd}, \\
\\
0 & \rm{otherwise.}
\end{cases}\]
Comparing with the coefficient of $\mf{C}^1_s$ in Proposition~\ref{prop:limsym} we conclude the proof of the proposition.
\end{proof}

\subsection{General matrices}\label{subsec:limgeneral}

For positive integers $m\geq n$ and for $s=0,\cdots,n$, we let
\begin{equation}\label{eq:defAsn}
 \mc{A}(s;m,n)=\{\ll\in\bb{Z}^n_{\dom}:\ll_s\geq s+m-n,\ll_{s+1}\leq s\}.
\end{equation}
If $\ll\in\mc{A}(s;m,n)$ then we define a dominant weight $\ll(s)\in\bb{Z}^m_{\dom}$ by
\begin{equation}\label{eq:deflls}
 \ll(s) = (\ll_1-(m-n),\cdots,\ll_s-(m-n),\underbrace{s,\cdots,s}_{m-n},\ll_{s+1},\cdots,\ll_n)
\end{equation}
For vector spaces $W_1$, $W_2$, with $\dim(W_1)=m$, $\dim(W_2)=n$, and $s=0,\cdots,n$ we define $\mf{A}_s\in\G(W_1,W_2)$~by
\begin{equation}\label{eq:defas}
\mf{A}_s=\bigoplus_{\ll\in\mc{A}(s;m,n)} S_{\ll(s)}W_1\oo S_{\ll}W_2. 
\end{equation}

\begin{proposition}\label{prop:limsgeneral}
 We write $W=W_1\oo W_2$. If $E=\det(W)\oo\Sym(W)$ then for $k=0,\cdots,n,$
\[(-1)^{k\cdot(m-n)}\cdot\left(\lim_{r\to\infty} p_{k,r}(V_1)\oo p_{k,r}(V_2)\oo E\right)=\sum_{s=n-k}^n (-1)^{(m-n)\cdot(n-k-s)}\cdot {s\choose s-n+k}\cdot \mf{A}_s.\]
\end{proposition}

\begin{proof}[Proof of Proposition~\ref{prop:limsgeneral}]
 Consider dominant weights $\tl{\d}\in\bb{Z}^m_{\dom}$ and $\tl{\ll}\in\bb{Z}^n_{\dom}$, and let 
\begin{equation}\label{eq:tldtlll}
\d=\tl{\d}-(n^m),\quad\ll=\tl{\ll}-(m^n).
\end{equation}
We obtain using (\ref{eq:Sll+rI=Sll(rI)}), (\ref{eq:genlPieri}), and easy manipulations that $\scpr{S_{\tl{\d}}W_1\oo S_{\tl{\ll}}W_2}{p_{k,r}(V_1)\oo p_{k,r}(V_2)\oo E}$ equals
\[\sum_{I\in{[m]\choose k},\ J\in{[n]\choose k}}\sgn(\s(I))\cdot\sgn(\s(J))\cdot\scpr{S_{\d(r,I)}W_1\oo S_{\ll(r,J)}W_2}{\Sym(W)}.\]
Using (\ref{eq:defll12}) and writing $\mu | (0^{m-n})$ for the sequence obtained by appending $m-n$ zeros to $\mu$, we get for~$r\gg 0$
\[
 \scpr{S_{\d(r,I)}W_1\oo S_{\ll(r,J)}W_2}{\Sym(W)} = \begin{cases}
 1, & \rm{if }\d^1(I)=\ll^1(J),\ \d^2(I)=\ll^2(J) | (0^{m-n}),\rm{ and }\d^2(I)\in\bb{Z}^{m-k}_{\geq 0}, \\
 0, & \rm{otherwise}.
 \end{cases}
\]
Let $u\in\{0,\cdots,m\}$ be the unique index such that $\d_u\geq u-m\geq\d_{u+1}$. The condition $\d^2(I)\in\bb{Z}^{m-k}_{\geq 0}$ is equivalent to the inclusion $\{u+1,\cdots,m\}\subset I$, which implies $u\geq m-k$. When $m>n$, the last $m-n$ entries of $\d^2(I)$ being $0$ forces $\d_u=\d_{u-1}=\cdots=\d_{u-m+n+1}=u-m$, and all the elements $u,u-1,\cdots,u-m+n+1$ to be contained in $I^c=[m]\setminus I$. We modify $\d$ and $I$ as follows: we consider $\ol{\d}\in\bb{Z}^n_{\dom}$ and $\ol{I}\in{[n]\choose k}$ defined by 
\[\ol{\d}=(\d_1,\cdots,\d_{u-m+n},\d_{u+1}-(m-n),\cdots,\d_m-(m-n)),\]
\[\quad\ol{I}=\{i_1,\cdots,i_{k-m+u},u+1-(m-n),u+2-(m-n),\cdots,n\},\]
so that $\ol{I}^c=[n]\setminus\ol{I}=I^c\setminus\{u,u-1,\cdots,u-m+n+1\}$. The conditions $\d^1(I)=\ll^1(J)$ and $\d^2(I)=\ll^2(J) | (0^{m-n})$ are then equivalent to $\ol{\d}^1(\ol{I})=\ll^1(J)$ and $\ol{\d}^2(\ol{I})=\ll^2(J)$. Since both $\ll,\ol{\d}$ are dominant weights, these equalities can only hold for $\ol{\d}=\ll$ and $\ol{I}=J$. Note that the freedom in choosing $I$ (or $\ol{I}=J$) is in the choice of an increasing sequence $i_1<\cdots<i_{k-m+u}$ inside $\{1,\cdots,u\}$, i.e. there are ${u-m+n\choose k-m+u}$ choices for $I$ once we fix $\d$. Writing $s=u-m+n$ we get
\[s\geq(m-k)-m+n=n-k,\]
\[\ll_s=\d_s=\d_{u-m+n}\geq\d_u\geq u-m=s-n,\]
\[\ll_{s+1}=\d_{u+1}-(m-n)\leq(u-m)-(m-n)=(s-n)-(m-n)=s-m,\]
and moreover
\[\d=(\ll_1,\cdots,\ll_s,\underbrace{s-n,\cdots,s-n}_{m-n},\ll_{s+1}+(m-n),\cdots,\ll_n+(m-n)).\]
It follows (using (\ref{eq:deflls}), (\ref{eq:tldtlll})) that $\tl{\d}=\tl{\ll}(s)$. Since $\sgn(\s(I))=(-1)^{(m-n)\cdot(m-u)}\cdot\sgn(\s(\ol{I}))$, it follows that if $\ol{I}=J$ and $m-u=n-s$ then
\[\sgn(\s(I))\cdot\sgn(\s(J))=(-1)^{(m-n)\cdot(n-s)}.\]
Putting everything together, and using ${u-m+n\choose k-m+u}={s\choose s-n+k}$, we obtain for $r\gg 0$
\[\scpr{S_{\tl{\d}}W_1\oo S_{\tl{\ll}}W_2}{p_{k,r}(V_1)\oo p_{k,r}(V_2)\oo E}=\begin{cases}
 \displaystyle(-1)^{(m-n)\cdot(n-s)}\cdot{s\choose s-n+k} & \rm{if }\tl{\ll}\in\mc{A}(s;m,n)\rm{ and }\tl{\d}=\tl{\ll}(s)\\
 & \rm{for some }s\geq n-k,\\
 0 & \rm{otherwise.}
\end{cases}
\]
Multiplying by $(-1)^{k\cdot(m-n)}$ and taking the limit $r\to\infty$ yields the desired conclusion.
\end{proof}

\subsection{Skew-symmetric matrices}\label{subsec:limskew}

For a positive integer $m$ and for $s=0,\cdots,m$, we let
\begin{equation}\label{eq:defmcBsn}
\begin{aligned}
\mc{B}(s,2m) &= \{\ll\in\bb{Z}^{2m}_{\dom}:\ll_{2s}\geq(2s-1),\ll_{2s+1}\leq 2s,\ll_{2i-1}=\ll_{2i}\rm{ for all }i\}, \\
\mc{B}(s,2m+1) &= \{\ll\in\bb{Z}^{2m+1}_{\dom}:\ll_{2s+1}=2s,\ll_{2i-1}=\ll_{2i}\rm{ for }i\leq s,\ll_{2i}=\ll_{2i+1}\rm{ for }i>s\}.
\end{aligned}
\end{equation}
For a vector space $W$ with $\dim(W)=n$, and for $s=0,\cdots,m=\lfloor n/2\rfloor$ we define $\mf{B}_s\in\G(W)$ via
\begin{equation}\label{eq:defbs}
\mf{B}_s=\bigoplus_{\ll\in\mc{B}(s,n)} S_{\ll}W. 
\end{equation}

\begin{proposition}\label{prop:limskew}
 If $E=\det\left(\bw^2 W\right)\oo\Sym\left(\bw^2 W\right)$ and $k=0,\cdots,m=\lfloor n/2\rfloor$, then
\[\lim_{r\to\infty} p_{2k,r}(V)\oo E=\sum_{s=m-k}^m {s\choose m-k}\cdot \mf{B}_s.\]
\end{proposition}

With $m=\lfloor n/2\rfloor$ and $1\leq k\leq m$, we define the following collections of dominant weights
\begin{equation}\label{eq:defmcBs}
\begin{aligned}
 \mc{B} &=\{\ll\in\bb{Z}^{2k}_{\dom}:\ll_{2i-1}=\ll_{2i}\rm{ for }i=1,\cdots,k\}, \\
 \mc{B}_{\geq n-1} &=\{\ll\in\bb{Z}^{n-2k}_{\dom}:\ll_{2i-1}=\ll_{2i}\geq n-1\rm{ for }i=1,\cdots,m-k,\ll_{n-2k}=n-1\rm{ if }n\rm{ is odd}\}.
\end{aligned}
\end{equation}
We partition $\bb{Z}^n_{\dom}$ into the following collections of dominat weights $\mc{Y}(u)$, $u=0,\cdots,n$, defined by
\begin{equation}\label{eq:Yu}
 \mc{Y}(u)=\{\ll\in\bb{Z}^n_{\dom}:\ll_u\geq u-1\geq\ll_{u+1}\},
\end{equation}
In analogy with Lemma~\ref{lem:ll^iinmcC} one can prove:

\begin{lemma}\label{lem:ll^iinmcB}
 If $I\in{[n]\choose 2k}$, then the conditions $\ll^1(I)\in\mc{B}$ and $\ll^2(I)\in\mc{B}_{\geq n-1}$ are equivalent to
 \begin{itemize}
  \item $\ll\in\mc{Y}(u)$ for some $u=n-2k,\cdots,n$.
  \item $\{u+1,\cdots,n\}\subset I$.
  \item $\ll_{i_{2t-1}}=\ll_{i_{2t}}$ and $i_{2t}=i_{2t-1}+1$ for all $t=1,\cdots,m$. If $n$ is odd then $u\in I^c$ is odd and $\ll_u=u-1$.
 \end{itemize}
\end{lemma}

\begin{lemma}\label{lem:permi2t-1i2t}
 Assume that $\ll,I$ satisfy the equivalent conditions in Lemma~\ref{lem:ll^iinmcB}. If $n=2m$ is even then
\begin{equation}\label{eq:permi2teven}
 \{(i_1,i_2),(i_3,i_4),\cdots,(i_{2m-1},i_{2m})\} = \{(1,2),(3,4),\cdots,(2m-1,2m)\}.
\end{equation}
If $n=2m+1$ is odd and if we write $u=2s+1$, then we have
\begin{equation}\label{eq:permi2todd}
 \{(i_1,i_2),(i_3,i_4),\cdots,(i_{2m-1},i_{2m})\} = \{(1,2),\cdots,(2s-1,2s),(2s+2,2s+3),\cdots,(2m,2m+1)\}.
\end{equation}
Moreover, we have that $\ll\in\mc{B}(s,n)$ for some $s=m-k,\cdots,m$.
\end{lemma}

\begin{proof}
 The conclusions (\ref{eq:permi2teven}--\ref{eq:permi2todd}) follow from the fact that $i_1,\cdots,i_n$ give a permutation of $[n]$ with $i_{2t}=i_{2t-1}+1$, and $i_n=u$ is odd when $n$ is odd. If $n$ is odd and $u=2s+1$, it follows from $u\geq n-2k$ that $s\geq m-k$. Moreover, we have $\ll_{2s+1}=\ll_u=u-1=2s$, and it follows from (\ref{eq:permi2todd}) that $\ll\in\mc{B}(s,n)$.
 
 Assume now that $n=2m$ is even. It follows from (\ref{eq:permi2teven}) that $i_n$ is even, so we can write $i_n=2s'$. Since $i_n+1,\cdots,n\in I$, we get $n-i_n\leq 2k$, i.e. $s'\geq m-k$. Since $i_n\leq u$, we have $\ll_{2s'}\geq\ll_u\geq u-1\geq 2s'-1$. We have by (\ref{eq:permi2teven}) that $\ll_{2i-1}=\ll_{2i}$ for $i=1,\cdots,m$, so taking $s$ to be the maximal index for which $\ll_{2s}\geq 2s-1$ we find that $s\geq s'$ and $\ll_{2s+1}=\ll_{2s+2}<2s+1$, i.e. $\ll\in\mc{B}(s,n)$.
\end{proof}

\begin{lemma}\label{lem:sizeBsn/2-sj}
 Let $m=\lfloor n/2\rfloor$, and for $m-k\leq s\leq m$ define the collection of partitions
\begin{equation}\label{eq:defBsn/2-sj}
\begin{aligned}
 \mc{B}(k,n/2-k,s)=\{\mu\in P(2k,n-2k):&\mu'_i\rm{ even for }i=1,\cdots,n-2k,\mu'_{n-2k}=2m-2s,\\
 &\mu_i\rm{ even for }i=2m-2s+1,\cdots,2k\}. 
\end{aligned}
\end{equation}
 Every partition $\mu\in\mc{B}(k,n/2-k,s)$ has even size, and the cardinality of the set $\mc{B}(k,n/2-k,s)$ is given by
\[
 |\mc{B}(k,n/2-k,s)| = \begin{cases}
 {s\choose m-k} & \rm{ if }n=2m+1\rm{ is odd}; \\                       
 {s-1\choose m-1-k} & \rm{ if }n=2m\rm{ is even}. \\                       
\end{cases}
\]
\end{lemma}

\begin{proof}
 Since each $\mu_i'$ is even, $|\mu|=|\mu'|$ is even. To compute the size of $\mc{B}(k,n/2-k,s)$ we first note that the condition $\mu'_{n-2k}=2m-2s$ implies that $\mu_1=\cdots=\mu_{2m-2s}=n-2k$, so any $\mu\in\mc{B}(k,n/2-k,s)$ is determined by $\ol{\mu}=(\mu_{2m-2s+1},\cdots,\mu_{2k})\in P(2(k+s-m),n-2k)$. Since $\mu'_{n-2k}=2m-2s$, we must have $\ol{\mu}_1<n-2k$. The condition $\mu\in\mc{B}(k,n/2-k,s)$ is then equivalent (using (\ref{eq:defPhjab})) to
 \[\ol{\mu}\in\begin{cases}
P^{0,0}(2(k+s-m),2(m-k)) & \rm{if }n\rm{ is odd}, \\
P^{0,0}(2(k+s-m),2(m-1-k)) & \rm{if }n\rm{ is even}. \\
\end{cases}
\]
By Lemma~\ref{lem:countpartitions}, the number of choices for $\ol{\mu}$ is ${s\choose m-k}$ if $n$ is odd, respectively ${s-1\choose m-1-k}$ if $n$ is even.
\end{proof}

\begin{lemma}\label{lem:corrImcBkn/2s}
 Assume that $\ll\in\mc{B}(s,n)$ for some $s\geq m-k$. The collection of subsets $I\in{[n]\choose 2k}$ for which $\ll^1(I)\in\mc{B}$ and $\ll^2(I)\in\mc{B}_{\geq n-1}$ corresponds via (\ref{eq:Itomu}) to $\mc{B}'(k,n/2-k,s)$, where
 \[
  \mc{B}'(k,n/2-k,s)=\begin{cases}
  \mc{B}(k,n/2-k,s) & \rm{if }n\rm{ is odd}, \\
  \displaystyle\bigcup_{s'=m-k}^s\mc{B}(k,n/2-k,s') & \rm{if }n\rm{ is even}.
 \end{cases}
 \]
\end{lemma}

\begin{proof} Consider $\ll\in\mc{B}(s,n)$ for $s\geq m-k$, and $I\in{[n]\choose 2k}$ satisfying the conditions of Lemma~\ref{lem:ll^iinmcB}. If $n=2m+1$ is odd then $i_n=2s+1$ and $I$ contains $2s+2,\cdots,n$, i.e. the corresponding $\mu\in P(2k,n-2k)$ has
\[\mu_1=\cdots=\mu_{2m-2s}=n-2k,\quad \mu_{2m-2s+1}<n-2k,\]
so $\mu'_{n-2k}=2m-2s$. For $2k<t<n$ we have that $i_t\leq 2s$, so $\mu'_{t-2k}=t-i_t$ is even by (\ref{eq:permi2todd}). The set of $\mu_i$ with $2m-2s<i\leq 2k$ coincides with that of differences $i_t-t$ for $1\leq t\leq 2(k-m+s)$, which are all even again by (\ref{eq:permi2todd}) and the fact that $i_t\leq 2s$ for $t\leq 2(k-m+s)$.
 
Assume next that $n=2m$ is even, and use (\ref{eq:permi2teven}) to write $i_n=2s'$. As in the previous paragraph, this implies $\mu'_{n-2k}=2m-2s'$. By (\ref{eq:permi2teven}) all the differences $t-i_t$ are even, so all $\mu_i,\mu'_i$ are even. This shows that $I\in\mc{B}(k,n/2-k,s')$. Since $i_n+1,\cdots,n\in I$ we get as before $s'\geq m-k$. If $s'>s$ then $\ll^2(I)_{n-2k}=\ll_{i_n}+n-i_n=\ll_{2s'}+n-2s'\leq\ll_{2s+1}+n-2s'\leq 2s+n-2s'\leq n-2$, a contradiction.

The verification that $\mu\in\mc{B}'(k,n/2-k,s)$ yields a subset $I$ with $\ll^1(I)\in\mc{B}$ and $\ll^2(I)\in\mc{B}_{\geq n-1}$ follows easily by tracing back the arguments.
\end{proof}

\begin{proof}[Proof of Proposition~\ref{prop:limskew}]
 We have
\[\scpr{S_{\ll}W}{p_{2k,r}(V)\oo E}=\scpr{S_{\ll}W\oo p_{2k,r}(W)}{E}\overset{(\ref{eq:genlPieri}),(\ref{eq:Sll+rI=Sll(rI)})}{=}\sum_{I\in{[n]\choose 2k}}\sgn(\s(I))\scpr{S_{\ll(r,I)}W}{E}.\]
Using notation (\ref{eq:defll12}) and (\ref{eq:defmcBs}) we get that for $r\gg 0$
\[
 \scpr{S_{\ll(r,I)}W}{E} = \begin{cases}
 1, & \rm{if }\ll^1(I)\in\mc{B}\rm{ and }\ll^2(I)\in\mc{B}_{\geq n-1}, \\
 0, & \rm{otherwise}.
 \end{cases}
\]

It follows that
\[\lim_{r\to\infty} p_{2k,r}(V)\oo E = \sum_{\substack{\ll\in\bb{Z}^n_{\dom},\ I\in{[n]\choose 2k}\\ \ll^1(I)\in\mc{B},\ \ll^2(I)\in\mc{B}_{\geq n-1}}}\sgn(\s(I))\cdot S_{\ll}W\] \[\overset{Lemmas~\ref{lem:permi2t-1i2t}-\ref{lem:corrImcBkn/2s}}{=}\sum_{s=m-k}^m\left(\sum_{\substack{\ll\in\mc{B}(s,n) \\ \mu\in\mc{B}'(k,n/2-k,s)}} (-1)^{|\mu|} S_{\ll}W\right)\overset{|\mu|\rm{ even}}{=} \sum_{s=m-k}^m\left(\sum_{\ll\in\mc{B}(s,n)} |\mc{B}'(k,n/2-k,s)|\cdot S_{\ll}W\right)\]

If $n=2m+1$ is odd, then $|\mc{B}'(k,n/2-k,s)|=|\mc{B}(k,n/2-k,s)|={s\choose m-k}$, and the desired equality follows. Similarly, when $n=2m$ is even, we get
\[|\mc{B}'(k,n/2-k,s)|=\sum_{s'=m-k}^s |\mc{B}(k,n/2-k,s')|=\sum_{s'=m-k}^s {s'-1\choose m-1-k}={s\choose m-k}.\qedhere\]
\end{proof}

\section{Equivariant $\D$-modules on symmetric matrices}\label{sec:symm}

In this section we compute the characters of the $\GL$-equivariant $\D$-modules on the vector space $M^{\symm}$ of symmetric $n\times n$ matrices. We let $W$ denote a complex vector space of dimension $n$, $V=W^*$, and we identify $\Sym^2 W$ with $M^{\symm}$, where squares $w^2$ correspond to matrices of rank one. If we write $\GL=\GL(W)$, and let $M^{\symm}_s$ denote the subvariety of matrices of rank at most $s$ then the main result of this section is:

\begin{theorem}\label{thm:symmetric}
 There exist $(2n+1)$ simple $\GL$-equivariant holonomic $\D$-modules on $M^{\symm}$, namely 
\[
 C^j_s=\begin{cases}
 \mc{L}(M^{\symm}_{n-s},M^{\symm}) & \rm{if }j\equiv s\ (\opmod\ 2), \\
 \mc{L}(M^{\symm}_{n-s},M^{\symm};1/2) & \rm{if }j\equiv s+1\ (\opmod\ 2), \\
\end{cases}
\rm{ for }s=0,\cdots,n-1,\ j=1,2,
\]
and $C_n^1=C_n^2=\mc{L}(\{0\},M^{\symm})$. For all $s,j$, the character of $C_s^j$ is $\mf{C}_s^j$ (as defined in (\ref{eq:defcjs})).
\end{theorem}

The remaining assertion of the Theorem on Equivariant $\D$-modules on Symmetric Matrices described in the Introduction is the identification $C_s^1=F_{s+1}/F_{s-1}$ for $s=0,\cdots,n$: its proof follows closely the proof of Theorem~\ref{thm:nxnbasicthm} in the next section, so we leave the details to the interested reader. The classification of $\GL$-equivariant holonomic simple $\D$-modules is explained in Theorem~\ref{thm:classificationDmods}, so we only need to check that $\mf{C}_s^j$ is the character of $C_s^j$. For $k=1,\cdots,n$, we consider the situation of Section~\ref{subsec:pfwdDmods}, with $X=X_k=\bb{G}(k,V)$ and $\mc{R},\mc{Q}$ as in (\ref{eq:tautGr}). We let $U=\Sym^2 V$, $\eta=\Sym^2\mc{Q}$. If we write $Y=Y_k$, $\pi=\pi_k$, then (\ref{eq:diagrGeneric}) becomes 
\begin{equation}\label{eq:diagrsym^2W}
\xymatrix{
Y_k=\rm{Tot}_{X_k}(\Sym^2\mc{Q}^*) \ar@{^{(}->}[r] \ar[dr]_{\pi_k} & \Sym^2 W\times\bb{G}(k,V) \ar[d] \\
 & \Sym^2 W \\
}
\end{equation}
Locally on $X_k$, $\mc{Q}^*$ trivializes to a vector space of dimension $k$, and $Y_k$ gets identified with the space of $k\times k$ symmetric matrices. We take $\mc{L}=(\det\mc{Q})^{\oo 2}$, consider its $\GL$-equivariant inclusion $\mc{L}\subset\Sym^k\eta$ and note that $\mc{L}$ is locally generated by the symmetric determinant. If we let $Y_k^0\subset Y_k$ be the open set defined locally by the non-vanishing of the determinant, $\mc{M}_k^0=\mc{O}_{Y_k^0}$ is a $\mc{D}_{Y_k}$-module. Note that $Y_k^0$ maps isomorphically via $\pi_k$ to the orbit of symmetric matrices of rank $k$. As a $\GL$-equivariant quasi-coherent sheaf on $X_k$

\begin{equation}\label{eq:defM0symm}
\mc{M}_k^0=\bigoplus_{\ll\in\bb{Z}^k_{\dom},\ll_i\rm{ even}}S_{\ll}\mc{Q}=\varinjlim_{r\equiv k+1\ (\opmod\ 2)} (\det\mc{Q})^{\oo r}\oo\det\left(\Sym^2\mc{Q}^*\right)\oo\Sym\left(\Sym^2\mc{Q}^*\right),
\end{equation}
so condition (\ref{eq:twolims}) is satisfied in our context. The Euler characteristic of the $\D$-module pushforward $\int_{\pi_k}\mc{M}_k^0$ is now easily computed as a consequence of Proposition~\ref{prop:EuleraslimPr} and of Remark~\ref{rem:Euleraslim}:
\begin{equation}\label{eq:Dmodpfwdsym0}
\chi\left(\int_{\pi_k}\mc{M}_k^0\right) = (-1)^{k\cdot(n-k)}\cdot\left(\lim_{\substack{r\to\infty \\ r\equiv k+1\ (\opmod\ 2)}} p_{k,r}(V)\oo\det\left(\Sym^2 W\right)\oo\Sym\left(\Sym^2 W\right)\right),
\end{equation}
which is evaluated explicitly in Proposition~\ref{prop:limsym}.

We next explain why $\mc{M}_k^0\oo\det(\mc{Q})$ also has the structure of a $\D_{Y_k}$-module. Consider the \'etale double cover $Y_k^{1/2}$ of $Y_k^0$ defined locally by the square-root of the symmetric determinant. The structure sheaf $\mc{O}_{Y_k^{1/2}}$ is naturally a $\D_{Y_k^0}$-module \cite{coutinho-levcovitz} and hence also a $\D_{Y_k}$-module. It contains $\mc{M}_k^0$ so we can define $\mc{M}_k^1$ as the cokernel of the inclusion $\mc{M}_k^0\subset\mc{O}_{Y_k^{1/2}}$. As a $\GL$-equivariant quasi-coherent sheaf on $X_k$, $\mc{M}_k^1$ is given~by
\begin{equation}\label{eq:defM1symm}
\mc{M}_k^1=\bigoplus_{\ll\in\bb{Z}^k_{\dom},\ll_i\rm{ odd}}S_{\ll}\mc{Q}=\mc{O}_{Y_k^0}\oo\det(\mc{Q}).
\end{equation}
It follows that $\mc{M}_k^1$ satisfies the setting of Proposition~\ref{prop:EuleraslimPr} with $\mc{L}'=\det(\mc{Q})$ so we can compute the Euler characteristic of its direct image via $\pi_k$ as
\begin{equation}\label{eq:Dmodpfwdsym1}
\chi\left(\int_{\pi_k}\mc{M}_k^1\right) = (-1)^{k\cdot(n-k)}\cdot\left(\lim_{\substack{r\to\infty \\ r\equiv k\ (\opmod\ 2)}} p_{k,r}(V)\oo\det\left(\Sym^2 W\right)\oo\Sym\left(\Sym^2 W\right)\right),
\end{equation}
which is evaluated in Proposition~\ref{prop:limsym}. We are now ready to prove the main result of this section:

\begin{proof}[Proof of Theorem~\ref{thm:symmetric}] The classification of simple $\D$-modules follows from Theorem~\ref{thm:classificationDmods}, so it remains to check that in $\G(W)$ we have the equalities $C_s^j=\mf{C}_s^j$ for $s=0,\cdots,n$ and $j=1,2$. The equalities (\ref{eq:Dmodpfwdsym0}--\ref{eq:Dmodpfwdsym1}) together with Proposition~\ref{prop:limsym} yield for $s=1,\cdots,n$ and $j=1,2$,
\[C_{n-s}^j=\mf{C}_{n-s}^j+\sum_{i=n-s+1}^n (a^s_i\cdot \mf{C}^1_i+b^s_i\cdot\mf{C}^2_i),\rm{ for some integers }a^s_i,b^s_i.\]
Since $C_n^j=\det(\Sym^2 W)\oo\Sym(\Sym^2 W)$ has character $\mf{C}_n^j$ (by Cauchy's formula \cite[Prop.~2.3.8]{weyman}), the equation above is also satisfied for $s=0$. The Fourier transform $\mc{F}$ permutes the modules $C_s^j$, and it takes
\[
\mc{F}(\mf{C}_s^1)=\mf{C}_{n-s}^1\rm{ for }s=0,\cdots,n\rm{ and }\quad\mc{F}(\mf{C}_s^2)=\mf{C}_{n-s-1}^2\rm{ for }s=0,\cdots,n-1.
\]
We can then apply Lemma~\ref{lem:uppertriangularFourier} to the poset $\mc{P}=\{(s,j):s=0,\cdots,n-1,j=1,2\}\cup\{(n,1)\}$ with the lexicographic ordering given by $(s,j)<(s',j')$ if and only if $s<s'$, or $s=s'$ and $j<j'$. We let $\mf{v}_{(s,j)}=\mf{C}_s^j$ and $v_{(s,j)}=C_s^j$, and conclude using Lemma~\ref{lem:uppertriangularFourier} that $C_s^j=\mf{C}_s^j$ for all $s=0,\cdots,n$ and $j=1,2$.
\end{proof}

\section{Equivariant $\D$-modules on $m\times n$ matrices}\label{sec:mxnmatrices}

In this section we compute the characters of the $\GL$-equivariant $\D$-modules on the vector space $M$ of $m\times n$ matrices, for $m\geq n$. We consider $W_1,W_2$ vector spaces of dimension $\dim(W_1)=m$, $\dim(W_2)=n$, let $V_i=W_i^*$, and identify $W=W_1\oo W_2$ with $M$, where tensor products $w_1\oo w_2$ correspond to matrices of rank one. If we write $\GL=\GL(W_1)\times\GL(W_2)$, let $M_s$ denote the subvariety of matrices of rank at most $s$, and recall the notation (\ref{eq:defas}) for the characters $\mf{A}_s$, then the main result of this section is:

\begin{chargenl*}
 The simple $\GL$-equivariant holonomic $\D$-modules on $M$ are $A_s=\mc{L}(M_{n-s},M)$, $s=0,\cdots,n$, and for each $s$ the character of $A_s$ is $\mf{A}_s$. When $m=n$, $A_s$ is as described in Theorem~\ref{thm:nxnbasicthm}, while for $m>n$ it can be expressed in terms of local cohomology:
 \begin{equation}\label{eq:Asloccoh}
A_s=\mc{H}^{1+s\cdot(m-n)}_{M_{n-1}}(M,\mc{O}_M)=\mc{H}^{\codim(M_{n-s})}_{M_{n-s}}(M,\mc{O}_M).
 \end{equation}
\end{chargenl*}

We only need to show that $\mf{A}_s$ is the character of $A_s$, and to prove Theorem~\ref{thm:nxnbasicthm}. The classification of $\GL$-equivariant holonomic simple $\D$-modules is explained in Theorem~\ref{thm:classificationDmods}, while (\ref{eq:Asloccoh}) follows by comparing $\mf{A}_s$ with the characters of local cohomology modules from \cite[Thm.~4.5]{raicu-weyman-witt} and \cite[Thm.~6.1]{raicu-weyman}. 

\begin{proof}[Proof of Theorem~\ref{thm:nxnbasicthm}] Let's assume for now that $\mf{A}_s$ is the character of $A_s$, and write $W_1=W_2=\bb{C}^n$. Using Cauchy's formula \cite[Cor.~2.3.3]{weyman}, we get an equality of $\GL$-representations
\[S_{\det}=\bigoplus_{\ll\in\bb{Z}^n_{\dom}}S_{\ll}W_1\oo S_{\ll}W_2=\bigoplus_{i=0}^n\mf{A}_i.\]
As in Example~\ref{ex:TactsCN}, this shows that $A_0,\cdots,A_n$ are the $\D$-module composition factors of $S_{\det}$, each appearing with multiplicity one. It remains to check that $A_s=F_s/F_{s-1}$ where $F_s=\langle\det^{-s}\rangle_{\D}$.

We prove by induction on $s$ that the $\D$-module composition factors of $F_s$ are $A_0,\cdots,A_s$, which is clearly true for $s=0$. Assume that $s>0$ and that the induction hypothesis is valid for $F_{s-1}$, so that $S_{\det}/F_{s-1}=\bigoplus_{i=s}^n\mf{A}_i$ as $\GL$-representations. We must then have for some $i\geq s$ an inclusion of $\D$-modules $A_i\subset S_{\det}/F_{s-1}$. Using the character description, $A_i$ must contain the class of $\det^{-i}$ inside the quotient $S_{\det}/F_{s-1}$, and therefore it must also contain the classes of $\det^{-i+1},\det^{-i+2},\cdots$. If $i>s$ this contradicts the formula for the character of $A_i$. We conclude that $i=s$ and that we have an inclusion $A_s\subset S_{\det}/F_{s-1}$. Since $A_s$ is simple, it is generated by the class of $\det^{-s}$, so the image of $A_s$ is $F_s/F_{s-1}$.
\end{proof}

We note that, just as in Remark~\ref{rem:bsato}, the strict inclusions $F_{i-1}\subsetneq F_i$, $i=1,\cdots,n$, in Theorem~\ref{thm:nxnbasicthm} combined with Cayley's identity show that the $b$-function of the generic determinant is $b_{\det}(s)=\prod_{i=1}^n(s+i)$.

We conclude by showing that $\mf{A}_s$ is the character of $A_s$. For $k=1,\cdots,n$, we consider the situation of Section~\ref{subsec:pfwdDmods}, with $X=X_k=\bb{G}(k,V_1)\times\bb{G}(k,V_2)$ and $\mc{R}_1,\mc{Q}_1,\mc{R}_2,\mc{Q}_2$ as in (\ref{eq:tautGr}). We let $U=V_1\oo V_2$, $\eta=\mc{Q}_1\oo\mc{Q}_2$, and write $Y=Y_k$, $\pi=\pi_k$ in (\ref{eq:diagrGeneric}). We note that locally on $X_k$, $\mc{Q}_1^*,\mc{Q}_2^*$ trivialize to vector spaces of dimension $k$, and $Y_k$ gets identified with the space of $k\times k$ matrices. We take the line bundle $\mc{L}=\det\mc{Q}_1\oo\det\mc{Q}_2$, consider its $\GL$-equivariant inclusion $\mc{L}\subset\Sym^k\eta$, and note that $\mc{L}$ is locally generated by the function that assigns to a matrix its determinant. If we let $Y_k^0\subset Y_k$ be the open set defined locally by the non-vanishing of the determinant, then as a $\GL$-equivariant quasi-coherent sheaf on $X_k$, $\mc{O}_{Y_k^0}$ is given~by

\[\mc{O}_{Y_k^0}=\bigoplus_{\ll\in\bb{Z}^k_{\dom}}S_{\ll}\mc{Q}_1\oo S_{\ll}\mc{Q}_2=\varinjlim_r \mc{L}^{\oo r}\oo\det\left(\mc{Q}_1^*\oo\mc{Q}_2^*\right)\oo\Sym\left(\mc{Q}_1^*\oo\mc{Q}_2^*\right),\]
so condition (\ref{eq:twolims}) is satisfied in our context. The Euler characteristic of the $\D$-module pushforward $\int_{\pi_k}\mc{O}_{Y_k^0}$ is now easily computed as a consequence of Propositions~\ref{prop:EuleraslimPr} and \ref{prop:limsgeneral}, and of Remark~\ref{rem:Euleraslim}:
\[
\begin{aligned}
\chi\left(\int_{\pi_k}\mc{O}_{Y_k^0}\right) &= (-1)^{k\cdot(m-n)}\cdot\lim_{r\to\infty} p_{k,r}(V_1)\oo p_{k,r}(V_2)\oo\det(W)\oo\Sym(W) \\
&=\sum_{s=n-k}^n (-1)^{(m-n)\cdot(n-k-s)}\cdot{s\choose s-n+k}\cdot \mf{A}_s.
\end{aligned}
\]
Since $\mc{O}_{Y_k^0}$ maps isomorphically via $\pi_k$ to the orbit of rank $k$ matrices in $M$, the conclusion that $\mf{A}_s$ is the character of $A_s$ follows as in the proof of Theorem~\ref{thm:symmetric} by the linear algebra trick in Section~\ref{subsec:linalg}.

\section{Equivariant $\D$-modules on skew-symmetric matrices}\label{sec:skew}

In this section we compute the characters of the $\GL$-equivariant $\D$-modules on the vector space of skew-symmetric $n\times n$ matrices. We let $W$ denote a complex vector space of dimension $n$, $V=W^*$, and we identify $\bw^2 W$ with the vector space $M^{\skew}$ of $n\times n$ skew-symmetric matrices, where exterior products $w_1\wedge w_2$ correspond to matrices of rank two. If we write $\GL=\GL(W)$, $m=\lfloor n/2\rfloor$, let $M^{\skew}_s$ denote the subvariety of matrices of rank at most $2s$, and recall the notation (\ref{eq:defbs}) for the characters $\mf{B}_s$ then we have:

\begin{charskewsym*}
 The simple $\GL$-equivariant holonomic $\D$-modules on $M^{\skew}$ are $B_s=\mc{L}(M_{m-s}^{\skew},M^{\skew})$, $s=0,\cdots,m$, and for each $s$ the character of $B_s$ is $\mf{B}_s$. If $n=2m+1$ is odd then for $s=1,\cdots,m$, $B_s$ can be described in terms of local cohomology:
 \begin{equation}\label{eq:Bsloccoh}
B_s=\mc{H}^{2s+1}_{M^{\skew}_{m-1}}(M^{\skew},\mc{O}_{M^{\skew}})=\mc{H}^{\codim(M_{m-s}^{\skew})}_{M_{m-s}^{\skew}}(M^{\skew},\mc{O}_{M^{\skew}}).
 \end{equation}
 If $n=2m$ is even, we let $\Pf$ be an equation defining the hypersurface $M^{\skew}_{m-1}$. We let $S$ denote the coordinate ring of $M^{\skew}$, and consider $F_s=\langle \Pf^{-2s}\rangle_{\D}$, the $\D$-submodule of the localization $S_{\Pf}$ generated by $\Pf^{-2s}$ for $s=0,\cdots,m$ (and $F_{-1}=0$). We have that $B_s=F_s/F_{s-1}$ for $s=0,\cdots,m$.
\end{charskewsym*}

The classification of $\GL$-equivariant holonomic simple $\D$-modules is explained in Theorem~\ref{thm:classificationDmods}, while the equality (\ref{eq:Bsloccoh}) follows from \cite[Theorem~5.5]{raicu-weyman-witt} and \cite[(1.4)]{raicu-weyman-loccoh}. When $n=2m$, we get that $B_s=F_s/F_{s-1}$ just as in the proof of Theorem~\ref{thm:nxnbasicthm}. Note that Cayley's identity shows that $b_{\Pf}(s)$ divides $\prod_{i=1}^m(s+2\cdot i-1)$, which in turn implies that $\langle \Pf^{-2i}\rangle_{\D}=\langle \Pf^{-2i+1}\rangle_{\D}$. The strict inclusions $F_{i-1}\subsetneq F_i$ then force $2\cdot i-1$ to be a root of $b_{\Pf}(s)$ for $i=1,\cdots,m$, so in fact $b_{\Pf}(s)=\prod_{i=1}^m(s+2\cdot i-1)$.

To prove the theorem, it remains to check that $\mf{B}_s$ is the character of~$B_s$. For $k=1,\cdots,m$, we consider the situation of Section~\ref{subsec:pfwdDmods}, with $X=X_k=\bb{G}(2k,V)$ and $\mc{R},\mc{Q}$ as in (\ref{eq:tautGr}). We let $U=\bw^2 V$, $\eta=\bw^2\mc{Q}$, and write $Y=Y_k$, $\pi=\pi_k$ in (\ref{eq:diagrGeneric}). Locally on $X_k$, $\mc{Q}^*$ trivializes to a vector space of dimension $2k$, and $Y_k$ gets identified with the space of $2k\times 2k$ skew-symmetric matrices. We take the line bundle $\mc{L}=\det\mc{Q}$ to be the Pl\"ucker line bundle on $X$, consider its $\GL$-equivariant inclusion $\mc{L}\subset\Sym^k\eta$, and note that $\mc{L}$ is locally generated by the function that assigns to a skew-symmetric matrix its Pfaffian. If we let $Y_k^0\subset Y_k$ be the open set defined locally by the non-vanishing of the Pfaffian, then we get using Cauchy's formula \cite[Prop.~2.3.8]{weyman} that condition (\ref{eq:twolims}) is satisfied. As a consequence of Propositions~\ref{prop:EuleraslimPr} and \ref{prop:limskew}, and of Remark~\ref{rem:Euleraslim} we obtain
\begin{equation}\label{eq:Dmodpfwdskew}
\chi\left(\int_{\pi_k}\mc{O}_{Y_k^0}\right) = \lim_{r\to\infty} p_{2k,r}(V)\oo\det\left(\bw^2 W\right)\oo\Sym\left(\bw^2 W\right)=\sum_{s=m-k}^m {s\choose m-k}\cdot \mf{B}_s. 
\end{equation}
Since $\mc{O}_{Y_k^0}$ maps isomorphically via $\pi_k$ to the orbit of rank $2k$ matrices in $M^{\skew}$, we conclude as in the proof of Theorem~\ref{thm:symmetric} that $\mf{B}_s$ is the character of $B_s$ for all $s$.

\section{The simple regular holonomic $\D$-modules on rank stratifications}\label{sec:simplesrankstrat}

We let $X$ denote any of the vector spaces of general, symmetric, or skew-symmetric matrices, with the natural group action by row and column operations of the corresponding group $G$ as considered in the previous sections. We denote by $\Lambda$ the union of conormal varieties to the orbits of $G$, and consider the category $\mc{C}=\rm{mod}_{\Lambda}^{rh}(\D_X)$ of regular holonomic $\D_X$-modules whose characteristic variety is contained in $\Lambda$. The goal of this section is to describe explicitly the simple objects in $\mc{C}$ and obtain as a corollary a direct proof of Levasseur's conjecture \cite[Conj.~5.17]{levasseur} in the case of general and skew-symmetric matrices.

Via the Riemann--Hilbert correspondence, the simple objects in $\mc{C}$ are classified by irreducible local systems on the $G$-orbits. When the local systems are $G$-equivariant, the corresponding $\D_X$-modules have been described in the previous sections. The only orbits with irreducible non-equivariant local systems are the orbits $O\subset X$ of rank $n$ matrices, when $X$ is the vector space of $n\times n$ general or symmetric matrices, or when $X$ is the vector space of $2n\times 2n$ skew-symmetric matrices. In each of these cases, the complement of $O$ in $X$ is defined by a single polynomial $f$ which is the determinant of the generic (symmetric) $n\times n$ matrix in the first two cases, and it is the Pfaffian of the generic $2n\times 2n$ skew-symmetric matrix in the last case. The fundamental group of $O$ is equal to $\bb{Z}$, so the monodromy of the corresponding local system is given by a non-zero complex number $\ll=e^{2\pi i\a}$ with $\a\in\bb{C}/\bb{Z}$. We let $S$ denote the coordinate ring of $X$ and for $\a\in\bb{C}$ we consider the $\D_X$-module $F_{\a}=S_f\cdot f^{\a}$ (which only depends on the class of $\a$ in $\bb{C}/\bb{Z}$).

\begin{theorem}\label{thm:nonequivLa}
 With notation as above, consider the irreducible local system $L_{\a}$ on $O$ whose monodromy is given by $\ll=e^{2\pi i\a}$. If $L_{\a}$ is not $G$-equivariant then the corresponding simple object in $\operatorname{mod}_{\Lambda}^{rh}(\D_X)$ is $F_{\a}$.
\end{theorem}

\begin{proof} The restriction of $F_{\a}$ to $O$ is a rank one integrable connection whose corresponding local system has monodromy given by $\ll=e^{2\pi i\a}$. It follows that in order to prove the theorem we need to check that $F_{\a}$ is a simple $\D_X$-module. The condition that $L_{\a}$ is not $G$-equivariant is equivalent to (see Theorems~\ref{thm:equivRH} and~\ref{thm:classificationDmods})
\begin{itemize}
 \item $\a\notin\bb{Z}$ if $X$ is the space of general or skew-symmetric matrices.
 \item $\a\notin\frac{1}{2}\bb{Z}$ if $X$ is the space of symmetric matrices.
\end{itemize}
From now on we assume that $L_{\a}$ is not $G$-equivariant. It follows from the Cayley's identity (and its symmetric and skew-symmetric versions) that $F_{\a}$ is generated as a $\D_X$-module by $f^{\a}$ (or by $f^{r+\a}$ for any $r\in\bb{Z}$). In order to prove that $F_{\a}$ is simple, it is then sufficient to show that any non-zero $\D_X$-submodule $F\subset F_{\a}$ contains $f^{r+\a}$ for $r\gg 0$. Fix any such $F$.

We write $\mf{g}$ for the Lie algebra of $G$, and note that any $\D_X$-module is a $\mf{g}$-representation. In particular this is true about $F\subset F_{\a}$. Since $F_{\a}$ has a multiplicity free decomposition into irreducible $\mf{g}$-representations of the form $M\cdot f^{\a}$, where $M\subset S_f$ is an irreducible integral $\mf{g}$-representation, we may assume that $F$ contains one such $M\cdot f^{\a}$. Replacing $\a$ by $\a-r$ and $M$ by $M\cdot f^r$ for $r\in\bb{Z}$, we may assume that $M\subset S$. Since $M$ generates a non-zero ideal which is invariant under the action of $G$, it defines set-theoretically a proper closed $G$-invariant subset of $X$, which is necessarily contained in the zero locus of $f$ (the complement of $f$ is a dense orbit for the $G$-action). We obtain that the ideal in $S$ generated by $M$ contains all large enough powers of $f$, and therefore that $F$ contains $f^{r+\a}$ for $r\gg 0$, which concludes the proof of the theorem.
\end{proof}

We end by remarking that Theorem~\ref{thm:nonequivLa} yields a proof of Levasseur's conjecture in the case of general and skew-symmetric matrices. We have already seen that the irreducible $G$-equivariant local systems on the orbits of the group action give rise to simple $\D_X$-modules containing (and hence generated by) non-zero sections invariant under the action of the derived subgroup $G'$. By Theorem~\ref{thm:nonequivLa}, the remaining simple objects of $\mc{C}$ are all of the form $F_{\a}=S_f\cdot f^{\a}$. Since $f$ is a $G'$-invariant, the same is true about $f^{\a}$, so $F_{\a}$ contains non-zero $G'$-invariant sections.

\section*{Acknowledgments} 
I am grateful to Nero Budur, David Eisenbud, Sam Evens, Mircea Musta\c t\u a, Uli Walther and Jerzy Weyman for interesting conversations and helpful advice, as well as to the anonymous referee for suggesting many improvements to the presentation. Experiments with the computer algebra software Macaulay2 \cite{M2} have provided numerous valuable insights. This work was supported by the National Science Foundation Grant No.~1458715.


\end{document}